\documentclass[a4paper,12pt, reqno]{amsart} 
\usepackage{amssymb,amsthm,amsmath}
 \usepackage{multirow}
\usepackage{enumerate}
\usepackage{ifthen}
\usepackage{graphicx} 
\usepackage{tkz-euclide}
\usepackage[colorlinks,citecolor=blue,hypertexnames=false]{hyperref}

 \numberwithin{equation}{section}

\usepackage{amsfonts}
\usepackage{amscd}
\usepackage{amssymb}
\usepackage{enumerate}
\usepackage{graphicx}
\allowdisplaybreaks
\usepackage{mathabx}
\usepackage{color}
\usepackage{amsbsy}
\usepackage{graphicx}
\usepackage{amsthm}
\usepackage{amsmath}
\usepackage{amsxtra}
\usepackage{mathrsfs}
\usepackage{bbm}
\usepackage{dsfont}

\usepackage{ifthen}
 
\newtheorem{theorem}{Theorem}[section]

\theoremstyle{definition}
\newtheorem{defn}[theorem]{Definitions}
\theoremstyle{remark}
\newtheorem{rem}[theorem]{Remark}
\numberwithin{equation}{section}
\definecolor{red}{rgb}{1.0, 0.0, 0.0}
\newcommand{\Bea}{\begin{eqnarray*}}
	\newcommand{\Eea}{\end{eqnarray*}}
\newcommand{\Be} {\begin{equation*}}
	\newcommand{\Ee} {\end{equation*}}
\newcommand{\be} {\begin{equation}}
	\newcommand{\ee} {\end{equation}}
\newcommand{\bea} {\begin{eqnarray}}
	\newcommand{\eea} {\end{eqnarray}}

\newcommand{\HH}{\mathbb{H}^n}

{\qed\bigskip}

\newcounter{alphabet}

\ifx\undefined\bysame
\newcommand{\bysame}{\leavevmode\hbox to3em{\hrulefill}\,}
\fi
\usetikzlibrary{patterns}

\title[Sobolev spaces of negative order]
{Semilinear damped wave equations on the Heisenberg group  with initial  data from Sobolev spaces of negative order}
\author{Aparajita Dasgupta} 
\address{Aparajita Dasgupta \endgraf Department of Mathematics
	\endgraf Indian Institute of Technology  Delhi
	\endgraf Delhi, 110016  India.} 
\email{adasgupta@maths.iitd.ac.in}
\author{Vishvesh Kumar} 

\address{Vishvesh Kumar  \endgraf Department of Mathematics: Analysis, Logic and Discrete Mathematics
	\endgraf Ghent University
	\endgraf Krijgslaan 281, Building S8,	B 9000 Ghent,
	Belgium.} 
\email{Vishvesh.Kumar@UGent.be and vishveshmishra@gmail.com}

\author{Shyam Swarup Mondal} 

\address{Shyam Swarup Mondal   
	\endgraf Department of Mathematics: Analysis, Logic and Discrete Mathematics
	\endgraf Ghent University
	\endgraf Krijgslaan 281, Building S8,	B 9000 Ghent
	Belgium} 
\email{mondalshyam055@gmail.com}

\author{Michael Ruzhansky} 
\address{Michael Ruzhansky  \endgraf Department of Mathematics: Analysis, Logic and Discrete Mathematics
	\endgraf Ghent University
	\endgraf Krijgslaan 281, Building S8,	B 9000 Ghent
	Belgium
	\endgraf and
	\endgraf School of Mathematical Sciences
	\endgraf Queen Mary University of London, United Kingdom} 
\email{michael.ruzhansky@ugent.be}

\date{\today}
\begin{document}
	\allowdisplaybreaks

	\begin{abstract} 
		
		In this paper, we focus on studying the Cauchy problem for semilinear damped wave equations  involving the sub-Laplacian $\mathcal{L}$ on the Heisenberg group $\HH$ with power type nonlinearity $|u|^p$  and  initial data taken from  Sobolev spaces of negative order
		homogeneous Sobolev space $\dot H^{-\gamma}_{\mathcal{L}}(\HH), \gamma>0$, on $\HH$. In particular,  in the framework of Sobolev spaces of negative order,  we prove that the critical exponent is the exponent   $p_{\text{crit}}(Q, \gamma)=1+\frac{4}{Q+2\gamma},$ for some  $\gamma\in (0, \frac{Q}{2})$,   where $Q:=2n+2$ is the homogeneous dimension of $\HH$. More precisely,  we establish  
		\begin{itemize}
			\item  a  global-in-time existence of small data Sobolev solutions of lower regularity for $p>p_{\text{crit}}(Q, \gamma)$  in the energy evolution space;
				\item a  finite time blow-up of weak solutions for $1<p<p_{\text{crit}}(Q, \gamma)$ under certain conditions on the initial data by using the test function method.
		\end{itemize}  
		Furthermore, to precisely characterize the blow-up time, we derive sharp upper bound and lower bound estimates for the lifespan in the subcritical case.

	\end{abstract}
	\maketitle
	\tableofcontents 
	\section{Introduction and discussion on main results}\label{sec1}
	\subsection{Description of problem and background}
	In this study, our main aim  is to  determine a new critical exponent for the Cauchy problem  for a semilinear damped wave equation with the power  type nonlinearities as follows:
	\begin{align} \label{eq0010}
		\begin{cases}
			u_{tt}-\mathcal{L}u +u_{t} =|u|^p, & g\in \mathbb{H}^n,~t>0,\\
			u(0,g)=\varepsilon u_0(g),  & g\in \mathbb{H}^n,\\ u_t(0, g)=\varepsilon u_1(g), & g\in \mathbb{H}^n,
		\end{cases}
	\end{align}
	where  $\mathcal{L}$ is the sub-Laplacian on the Heisenberg group $\HH$, $1<p<\infty$, and the initial data $(u_0, u_1)$ with its size parameter $\varepsilon>0$  belongs to subelliptic (or Folland-Stein) homogeneous Sobolev spaces of negative order $\left(u_0, u_1\right) \in  \dot {H}_\mathcal{L}^{-\gamma}(\HH) \times  \dot {H}_\mathcal{L}^{-\gamma}(\HH)$ with $\gamma>0$. In other words,  we study the global-in-time existence of small data solutions and the blow-up in finite time of solutions to the Cauchy problem (\ref{eq0010}).

	To discuss the classical Euclidean scenario, let us consider  the  semilinear   damped  wave equation   on $\mathbb{R}^n$    with the power  type nonlinearities as follows:   
	\begin{align} \label{eq38}
		\begin{cases}
			u_{tt}-\Delta u+	u_t=|u|^p,& x \in \mathbb{R}^n, ~t>0, \\
			u( 0,x)=u_0(x), & x \in \mathbb{R}^n,\\
			u_t( 0,x)=u_1(x), & x \in \mathbb{R}^n,
		\end{cases}
	\end{align}
	where $1<p<\infty$ and $\Delta$ is the Laplacian on $\mathbb{R}^n$.   When the initial data belongs additionally to $L^1$-space, the global existence or a  blow-up result to  (\ref{eq38}), depending on the critical exponent has been studied in \cite{IKeta and Tanizawa, Matsumura, Todorova, Zhang} and references therein.     The critical exponent signifies the threshold condition on the exponent $p $ for the global-in-time Sobolev solutions and the blow-up of local-in-time weak solutions with small data.   The critical exponent  for solutions to (\ref{eq38}) is  the so-called {\it Fujita exponent} given by $p_{\text{Fuji}}(n):=1+\frac{2}{n}$ (see \cite{IKeta and Tanizawa,Matsumura, Todorova,Zhang}). More precisely, 
	\begin{itemize}
		\item When $n=1, 2$, Matsumura  in his seminal paper \cite{Matsumura} proved the global-in-time existence of small data solutions for   $p>p_{\text {Fuji}}(n)$.  
		\item For any $n \geq 1$,  a  global existence for  $p>p_{\text {Fuji}}(n)$ (by assuming compactly supported initial data) and blow-up of the local-in-time solutions in the subcritical case $1<p<p_{\text {Fuji}}(n)$ was  explored  by  Todorova and Yordanov \cite{Todorova}. 
		\item  For $p=p_{\text {Fuji}}(n),$ the blow-up result  was obtained by Zhang \cite{Zhang}. 
	\end{itemize}    
	Later, Ikehata and Tanizawa \cite{IKeta and Tanizawa} removed the restriction of compactly supported data for  the supercritical case $p>p_{\text {Fuji}}(n)$. Moreover,   the sharp lifespan estimates for the Cauchy problem  \eqref{eq38} with additional $L^1$-data assumption are given by $$
	T_\varepsilon\left\{\begin{array}{ll}
		=\infty & \text { if } p>p_{\mathrm{Fuji}}(n), \\
		\simeq \exp \left(C \varepsilon^{-(p-1)}\right) & \text { if } p=p_{\mathrm{Fuji}}(n), \\
		\simeq C \varepsilon^{-\frac{2(p-1)}{2-n(p-1)}} & \text { if } p<p_{\mathrm{Fuji}}(n),
	\end{array}\right.
	$$
	where $C$ is a positive constant independent of $\varepsilon$. We cite \cite{Ikeda16,Ikeda19,Ikeda15, Lai19, Li1995} for a detailed study on the sharp lifespan estimates for the Cauchy problem  (\ref{eq38}). We also refer to the excellent book    \cite{Rei}  for the global-in-time small data solutions for the semilinear damped wave equations in the Euclidean framework.

	In recent years, considerable attention has been devoted by numerous researchers to finding new critical exponents for the classical semilinear damped wave equations on $\mathbb{R}^n$ in different contexts. For instance, considering the Cauchy problem (\ref{eq38}) with initial data additionally belonging to $L^m$-spaces with $m \in(1,2)$, the critical exponent is changed and the new modified Fujita exponent becomes  $$p_{\text {Fuji}}(n):=p_{\text {Fuji }}\left(\frac{n}{m}\right)=1+\frac{2 m}{n}.$$
	Unlike the     $L^1$-data case, here   with additional $L^m$-regular data,      the global-in-time  solution   exists  uniquely  at the critical point    $p_{\text {Fuji}}\left(\frac{n}{m}\right)=1+\frac{2 m}{n}$.  This is the main difference between   $L^1$ and $L^m$ regular data. We refer to \cite{Ikeda2019, Ikeda2002,Nakao93} and references therein  for a detailed study  related to the   critical exponent   $p_{\text {Fuji}}\left(\frac{n}{m}\right)$ for the solutions to  the semilinear wave equations with the  $L^m$-regular data. 
	
	The study of the semilinear damped wave equation has also been extended in the non-Euclidean framework. Several papers have studied nonlinear PDEs in non-Euclidean settings in the last decades.   	  For example, the semilinear wave equation with or without damping has been investigated for the   Heisenberg group in \cite{Vla, 24,30}.    In the case of graded groups, we refer to the recent works \cite{palmieri, gra1, 30, gra3}.  	We refer to \cite{garetto,27, 31, 28, BKM22, DKM23} concerning the  wave equation on compact Lie groups and \cite{AP,HWZ, AKR22,AZ} on a Riemannian symmetric space of non-compact type.   In particular, for  the  semilinear damped wave equation  on the Heisenberg group $\HH$, namely
	\begin{align} \label{eq00100}
		\begin{cases}
			u_{tt}-\mathcal{L}u +u_t =|u|^p, & g\in \mathbb{H}^n,~t>0,\\
			u(0,g)=u_0(g),  & g\in \mathbb{H}^n,\\ u_t(0,g)=u_1(g), & g\in \mathbb{H}^n,
		\end{cases}
	\end{align}
	where   $p>1$, $\mathcal{L}$ is the sub-Laplacian on the Heisenberg group $\HH$,    it was shown in \cite{Vla} that    the critical exponent is given by 
	\begin{align}\label{eq36}
		p_{\text{Fuji}}(Q):=1+\frac{2}{Q},
	\end{align}
	where $Q:=2n+2$ represents the homogeneous dimension of $\HH$. 
	It is interesting to note that (\ref{eq36})
	is also   the {\it Fujita exponent}  for the semilinear heat equation
	\begin{align*}
		\begin{cases}
			u_t-\mathcal{L}u  =|u|^p, & g\in \mathbb{H}^n,~t>0,\\
			u(0,g)=u_0(g),  & g\in \mathbb{H}^n,\\  
		\end{cases}
	\end{align*}
	on $\HH,$ where $p>1$. This topic has been discussed in \cite{Vla,palmieri,Nurgi}. However, in the case of a compact Lie group $\mathbb{G}$, it has  been shown in \cite{27} that $p_{\text {Fuji}}(0):=\infty$ is the critical exponent  for the solution to the semilinear {   damped} wave equation on $\mathbb{G}$. We cite \cite{30} for a global existence result with small data  in the more general setting of graded Lie groups for the semilinear damped wave equation involving a Rockland operator with an additional mass term.

	Recently, Chen and  Reissig \cite{Reissig} considered the following semilinear damped wave equation
	\begin{align} \label{eq40}
		\begin{cases}
			u_{tt}-\Delta u+	u_t= |u|^p,& x \in \mathbb{R}^n, ~t>0, \\
			u( 0, x)=u_0(x), & x \in \mathbb{R}^n,\\
			u_t( 0,x)=u_1(x), & x \in \mathbb{R}^n,
		\end{cases}
	\end{align}
	on $\mathbb{R}^n,$ where $p>1$    and the initial data additionally belonging to homogeneous Sobolev spaces of negative order  $\dot H^{-\gamma}(\mathbb{R}^n)$  with $\gamma>0$. They  obtained a new critical exponent $p = p_{\text{crit}}(n, \gamma) := 1+ \frac{4}{n+2\gamma}$ for some  $\gamma\in (0, \frac{n}{2})$  in this   framework. More specifically,   the authors  proved that:
	\begin{itemize}
		\item For $p>p_{\text{crit}}(n, \gamma) $,  the problem (\ref{eq40}) admits a global-in-time Sobolev solution for  sufficiently  small data   of lower regularity.
		\item For  $1<p<p_{\text{crit}}(n, \gamma) $,    the solutions to  (\ref{eq40})  blow-up in a finite time. In other words, there exists $T>0$ such that  the solution to  (\ref{eq40})  satisfies $\left\|u\left(\cdot, t_m\right)\right\|_{\infty} \rightarrow \infty$ as $t_m \rightarrow T$. 
	\end{itemize}   
	Further, 	   the authors also investigated   sharp lifespan estimates for weak solutions to (\ref{eq40}), in which the sharpness  of lifespan  estimates is given by
	$$
	T_\varepsilon\left\{\begin{array}{ll}
		=\infty & \text { if } p>p_{\text {crit }}(n, \gamma), \\
		\simeq C \varepsilon^{-\frac{2}{2 p^{\prime}-2-\frac{n}{2}-\gamma}} & \text { if } p<p_{\text {crit }}(n, \gamma),
	\end{array}\right.
	$$
	where $C$ is a positive constant independent of $\varepsilon$ and  $p^{\prime}$ is the Lebesgue exponent conjugate of $p$ such that $\frac{1}{p}+\frac{1}{p'}=1$. {   We remark that some other technical assumptions on $p$ and $\gamma$ are also required to find
	the sharp lifespan   estimate in the sub-critical case.}

	To the best of our knowledge, in  the non-Euclidean framework,  the subelliptic damped wave equation on the Heisenberg group with initial data localized in Sobolev spaces of negative order  has not been considered in the literature so far, even for the linear Cauchy problem. Therefore, an interesting and viable problem is to study several qualitative properties  such as global-in-time well-posedness, blow-up criterion, decay rate, asymptotic profiles to solutions for the subelliptic damped wave equations on the Heisenberg group $\HH$ with initial data  additionally belonging  to subelliptic Sobolev space $\dot {H}_\mathcal{L}^{-\gamma}(\HH)$ with $\gamma>0$.

	The main aim of this paper is to investigate and determine a critical exponent for the Cauchy problem for semilinear damped wave equation  (\ref{eq0010}) with the initial data additionally belonging to subelliptic homogeneous Sobolev spaces $  \dot  {H}_\mathcal{L}^{-\gamma}(\HH)$ of negative order {  $ -\gamma$.} More specifically, 
	\begin{itemize}
		\item  Under additional assumptions    for the initial data in $ \dot {H}_\mathcal{L}^{-\gamma}(\HH)$, we obtain  a new critical exponent to   (\ref{eq0010}) given by 
		\begin{align}\label{Critical-exponent}
			p_{\text {crit }}(Q, \gamma):=1+\frac{4}{Q+2 \gamma},
		\end{align}
		with $\gamma \in  {   \left(0,   \frac{\sqrt{Q^2+16 Q}-Q}{4}\right ). }$
		\item  We derive sharp lifespan estimates for weak solutions to the semilinear Cauchy problem (\ref{eq0010}). Define  	the lifespan  $T_\varepsilon$ as the maximal existence time for solution of (\ref{eq0010}), i.e.,
		\begin{align}\nonumber \label{eq42}
			T_\varepsilon :=\Big\{ T>0~&:~ \text{there exists a unique local-in-time solution to the Cauchy} \\& \text{  problem (\ref{eq0010}) on $[0, T)$   with a fixed parameter}~\varepsilon>0\Big\}.
		\end{align}
		If the initial data    is  from $\dot  {H}_\mathcal{L}^{-\gamma}(\HH)$ with {  $\gamma \in\left(0, \tilde{\gamma}\right)$} and the exponent   $p$ satisfies  $1+\frac{2 \gamma}{Q} \leq p \leq \frac{Q}{Q-2}$,  then the new sharp lifespan estimates are given by 
		$$
		T_\varepsilon\left\{\begin{array}{ll}
			=\infty & \text { if } p>p_{\text {crit }}(Q, \gamma), \\
			\simeq   C \varepsilon^{-\left(\frac{1}{p-1}-\left(\frac{Q}{4}+\frac{\gamma}{2}\right)\right)^{-1}} & \text { if } p<p_{\text {crit }}(Q, \gamma),
		\end{array}\right.
		$$
		where  the positive constant $C$ is independent of $\varepsilon.$
		
	\end{itemize}

	\subsection{Main results: Comprehensive review}
	Throughout the paper, we denote $L^{q}(\mathbb{H}^n), 1 \leq  q<\infty$, the space of $q$-integrable functions on $\HH$ with respect to  the Haar measure $dg$ on $\HH,$ which is nothing but the Lebesgue measure of $\mathbb{R}^{2n+1},$   and the space of all essentially bounded functions on $\HH$ for $q=\infty$.     The fractional subelliptic (Folland-Stein) Sobolev space $H_{\mathcal{L}}^s(\HH), s \in \mathbb{R}$, associated to the sub-Laplacian $\mathcal{L}$ on $\HH$, is defined as
	$$
	H_{\mathcal{L}}^s\left(\mathbb{H}^n\right):=\left\{f \in \mathcal{D}^{\prime}\left(\mathbb{H}^n\right):(I-\mathcal{L})^{s / 2} f \in L^2\left(\mathbb{H}^n\right)\right\},
	$$
	equipped  with the norm $$\|f\|_{H_{\mathcal{L}}^s\left(\mathbb{H}^n\right)}:=\left\|(I-\mathcal{L})^{s / 2} f\right\|_{L^2\left(\mathbb{H}^n\right)}.
	$$ 
	Similarly,  we denote by $ \dot{H}_{\mathcal{L}}^{ s}(\mathbb{H}^n),$ the  homogeneous Sobolev  defined as the space of all $f\in \mathcal{D}'(\HH)$ such that $(-\mathcal{L})^{{s}/{2}}f\in L^2(\HH)$. At times, we also write	$	H_{\mathcal{L}}^s$ and $	\dot{H}_{\mathcal{L}}^s$ for $H_{\mathcal{L}}^s\left(\mathbb{H}^n\right)$ and $\dot{H}_{\mathcal{L}}^s\left(\mathbb{H}^n\right),$ respectively from here on.

	For $(u_0, u_1)\in{  \mathcal{A}_{\mathcal{L}}^{s, -\gamma}}:= (H^s_{\mathcal{L}} \cap \dot {H}^{-\gamma}_\mathcal{L} ) \times (L^2 \cap \dot{H}^{-\gamma}_\mathcal{L} )$, we denote   $\left\|\left(u_0, u_1\right)\right\|_{{  \mathcal{A}_{\mathcal{L}}^{s, -\gamma}}}$ as 
	$$\left\|\left(u_0, u_1\right)\right\|_{{  \mathcal{A}_{\mathcal{L}}^{s, -\gamma}}}:=\|u_0\|_{H^s_{\mathcal{L}} \cap \dot{H}^{-\gamma}_\mathcal{L}}+\|u_1\|_{ L^2 \cap \dot {H}^{-\gamma}_\mathcal{L}  }.$$	Utilizing techniques derived from the  noncommutative Fourier analysis on the Heisenberg group $\HH$, we present our initial finding regarding time decay estimates in the $\dot{H}_\mathcal{L}^s$-norm of solutions to the homogeneous version of the Cauchy problem (\ref{eq0010}). This result is detailed below.  
	
	\begin{theorem}\label{eq0133} Let $\HH$ be the Heisenberg group with the homogeneous dimension $Q.$
		Let $\left(u_0, u_1\right) \in\left(H_\mathcal{L}^s \cap \dot {H}_\mathcal{L}^{-\gamma}\right) \times\left(H_\mathcal{L}^{s-1} \cap  \dot{H}_\mathcal{L}^{-\gamma}\right)$ such that  $s \geq 0$ and $\gamma \in \mathbb{R}$ such that $s+\gamma \geq  0$. 
		Consider  the following linear Cauchy problem   	\begin{align}\label{eq001000re}
			\begin{cases}
				u_{tt}-\mathcal{L}u+u_t =0, & g\in \mathbb{H}^n,t>0,\\
				u(0,g)=  u_0(g),  & g\in \mathbb{H}^n,\\ u_t(0,g)=  u_1(g), & g\in \mathbb{H}^n.
			\end{cases}
		\end{align}
		Then, the solution $u$ satisfies the following 
		$ \dot{H}_\mathcal{L}^s$-decay estimate	\begin{align}\label{hom}
			\|u(t, \cdot)\|_{ \dot {H}_\mathcal{L}^s} \lesssim(1+t)^{-\frac{s+\gamma}{2}}\left(\left\|u_0\right\|_{H^s_\mathcal{L} \cap\dot {H}^{-\gamma}_\mathcal{L}}+\left\|u_1\right\|_{H_\mathcal{L}^{s-1} \cap \dot{H}_\mathcal{L}^{-\gamma}}\right) ,
		\end{align}
		for any $t\geq 0$. 
	\end{theorem}
	The next result is about the global-in-time    well-posedness of the   Cauchy problem   (\ref{eq0010})  in the energy evolution space  $\mathcal C\left([0,T],  H^s_{\mathcal{L}}\right) , s\in(0, 1]$.   In this case, a version of a Gagliardo-Nirenberg type inequality on $\HH$ (see Section \ref{sec2}) will play a crucial role in estimating the power type nonlinearity in $L^2(\mathbb{H}^n)$. Let us first make clear what do we mean by a solution of (\ref{eq0010}). For the global existence result we will work with the mild solutions of (\ref{eq0010}).  
	
	We say that a function  $u$ is  a {\it mild solution} to (\ref{eq0010})  on $[0, T]$ if $u$ is a fixed point for  the       integral operator  $N: u \in X_s(T) \mapsto N u(t, g) ,$ given  by 
	\begin{align}\label{f2inr} 
		N u(t, g):=u^{\text{lin}}(t, g) + u^{\text{non}}(t, g),
	\end{align}	in the energy evolution space $X_s(T) \doteq \mathcal{C}\left([0, T], H_{\mathcal{L}}^{s}(\mathbb{H}^n)\right), s\in (0, 1]$, 
	equipped with the norm
	\begin{align*} 
		\|u\|_{X_s(T)}&:=\sup\limits_{t\in[0,T]}\left ( (1+t)^{\frac{\gamma}{2}} \|u(t,\cdot)\|_{L^2}+(1+t)^{\frac{s+\gamma}{2}}\|u(t,\cdot)\|_{ \dot H^s_{\mathcal{L}}}\right ),
	\end{align*}
	with $\gamma>0,$
	where $$u^{\text{lin}}(t, g)=  u_{0}(g) *_{(g)}  E_{0}(t, g)+ u_{1}(g) *_{(g)}  E_{1}(t, g)$$ is the solution to the corresponding linear Cauchy problem (\ref{eq001000re})  and
	$$ u^{\text{non}}(t, g)= \int_{0}^{t}|u(s, g)|^{p} *_{(g)}  E_{1}(t-s, g) \;ds.$$
	Here $*_{(g)}$ is  the group convolution product on $\HH$ with respect to the $g$ variable and  $E_{0}(t, g)$ and $E_{1}(t, g)$  represent   the fundamental solutions to the homogeneous problem (\ref{eq001000re}) with initial data $\left(u_{0}, u_{1}\right)=\left(\delta_{0}, 0\right)$ and $\left(u_{0}, u_{1}\right)=$ $\left(0, \delta_{0}\right)$, respectively. 
	
	We prove the global-in-time existence and uniqueness of small data Sobolev solutions to the semilinear damped wave equation (\ref{eq0010}) of low regularity by finding a unique fixed point to the operator $N$, i.e., $Nu\in X_s(T)$ for all $T>0$.  It means that there exists a unique global solution $u$    to the equation $N u=u\in X_s(T)$ which also gives the solution to (\ref{eq0010}). 	In order to prove that $N$ has a uniquely determined fixed point, we use Banach's fixed point argument  {  on the space  $X_s(T)$ defined above.    }

	Keeping $p_{\text {crit }}(Q, \gamma) =1+\frac{4}{Q+2\gamma}$ in mind,  we state  the following global-in-time existence result.
	\begin{theorem}\label{well-posed}  
		Let $s \in(0,1]$ and $\gamma \in\left(0, \frac{Q}{2}\right)$.  Assume that the exponent $p$ satisfies 	$ 1<p \leq \frac{Q}{Q-2 s}$  and 
		\begin{align}\label{eq24}
			p\left\{\begin{array}{ll}
				>p_{\text {crit }}(Q, \gamma) & \text { if } \gamma \leq \tilde{\gamma}, \\
				\geq 1+\frac{2 \gamma}{Q} & \text { if } \gamma>\tilde{\gamma}, 
			\end{array}\right. 
		\end{align}
		where $\tilde{\gamma}$ denotes the positive root of  the quadratic equation $2 \tilde{\gamma}^2+Q \tilde{\gamma}-2 Q=0$. 
{   Then, there exists a small positive constant $\varepsilon_0$ such that for any $\left(u_0, u_1\right) \in \mathcal{A}_{\mathcal{L}}^{s,-\gamma}$ satisfying $\left\|\left(u_0, u_1\right)\right\|_{\mathcal{A}_{\mathcal{L}}^{s,-\gamma}}=\varepsilon \in\left(0, \varepsilon_0\right]$, the Cauchy problem for the semilinear damped wave equation (\ref{eq0010}) has a uniquely determined Sobolev solution
	$$
	u \in \mathcal{C}\left([0, \infty), H_{\mathcal{L}}^s\right).
	$$}
Therefore, the lifespan of the solution is given by $T_\varepsilon=\infty$.   Moreover, the solution satisfies the following two estimates listed below:
		$$
		\|u(t, \cdot)\|_{L^2}  \lesssim (1+t)^{-\frac{\gamma}{2}}\left\|\left(u_0, u_1\right)\right\|_{{  \mathcal{A}_{\mathcal{L}}^{s, -\gamma}}}, $$ 
		and
		$$
		\|u(t, \cdot)\|_{\dot H^s_{\mathcal{L}}}  \lesssim (1+t)^{-\frac{s+\gamma}{2}}\left\|\left(u_0, u_1\right)\right\|_{{  \mathcal{A}_{\mathcal{L}}^{s, -\gamma}}} .
		$$
	\end{theorem} 	
	It is crucial to emphasize that there is no loss in the decay of the solution when transitioning from the linear to the nonlinear problem. In other words, the decay rates outlined in the preceding theorem align precisely with the decay rates established for the corresponding linearized damped wave equation, as presented in Theorem \ref{eq0133}. However, the restriction $	1< p\leq \frac{Q}{Q-2s}$  in the above theorem is necessary in order to apply a Gagliardo-Nirenberg type inequality on $\HH$.	 	
	
	\begin{rem} Some examples for the admissible range of exponents    $p$ for the global-in-time existence result in the low dimensions Heisenberg group $\HH$ with $n=1$ and $2$, that is,   $Q=4$ and $6$, respectively,  are as follows: 
		\begin{itemize} 
			\item When $Q=4$, we take $s \in(0,1]$ and $\gamma \in (0, 2 )$ and the exponent satisfies
			
			$$
			\begin{array}{l} \vspace{0.3cm}
				1+\frac{2}{2+ \gamma}<p \leq  \frac{2}{2- s} ~\quad \text { if } ~0<\gamma \leq \tilde{\gamma}={   \sqrt{5}-1}, \\ 
				1+\frac{ \gamma}{2} \leq  p \leq \frac{2}{2- s} ~\quad\quad \text { if }  ~ {   \sqrt{5}-1}<\gamma<2.
			\end{array}
			$$
			\item When $Q=6$, we take $s \in(0,1]$ and $\gamma \in (0, 3)$ and the exponent satisfies
			
			$$
			\begin{array}{l} \vspace{0.3cm}
				1+\frac{2}{3+ \gamma}<p \leq  \frac{3}{3- s} ~\quad \text { if } ~0<\gamma \leq \tilde{\gamma}={   \frac{\sqrt{33}-3}{2}}, \\ 
				1+\frac{ \gamma}{3} \leq  p \leq \frac{3}{3- s} ~\quad\quad \text { if }  ~ {   \frac{\sqrt{33}-3}{2}}<\gamma<3 .
			\end{array}
			$$
	\end{itemize}
	Note that  the positive root $\tilde{\gamma}$ of $2 \tilde{\gamma}^2+Q \tilde{\gamma}-2 Q=0$ is always strictly less than 2 for any homogeneous dimension $Q$.   
\end{rem} 


Based on the aforementioned illustrations, it is clear that the global existence result, as stated in Theorem \ref{well-posed}, is only pertinent to specific lower homogeneous dimensions. This constraint is due to the technical stipulation that  $1 < p \leq \frac{Q}{Q-2s}$.  For global existence results in higher homogeneous dimensions, one can study Sobolev solutions by considering initial data from subelliptic  Sobolev spaces with an appropriate degree of higher/large regularity.

Our next result is about a    blow-up (in-time) result in the subcritical case to the Cauchy problem  (\ref{eq0010}) for certain values of $p$ regardless of the size of the initial data. Before the blow-up result, we first introduce a suitable notion of a weak solution to the   Cauchy problem (\ref{eq0010}).
\begin{defn}
	For any $T>0$, 	 a  weak solution of the Cauchy problem (\ref{eq0010}) in $[0, T) \times \HH$ is a function $u \in L_{\text {loc }}^p\left([0, T) \times \HH\right)$ that   satisfies  the following  integral relation:
	\begin{align}\label{2999}\nonumber
		&	\int_0^T \int_{\HH} u(t, g)\left(\partial_t^2 \phi(t, g)-\mathcal{L} \phi(t, g)-\partial_t \phi(t, g)\right) \;d g \;d t -\varepsilon\int_{\mathbb{H}^n} u_0(g) \phi(0, g) \;d g\\&-\varepsilon \int_{\mathbb{H}^n} u_1(g) \phi(0, g) \;d g +\varepsilon\int_{\HH} u_0(g) \partial_t \phi(0, g) \;d g =	\int_0^T \int_{\HH}|u(t, g)|^p \phi(t, g) \;d g \;d t,
	\end{align}
	for any $\phi \in \mathcal{C}_{0}^{\infty}([0, T) \times \HH)$. 
	If $T=\infty$, we call $u$ to be a global-in-time weak solution to (\ref{eq0010}), otherwise   $u$ is said to be a local in time weak solution to (\ref{eq0010}).
\end{defn}

{   Let  $|\cdot|$   be any homogeneous norm on the Heisenberg group  $\mathbb{H}^n$, while we denote $\left(1+|g|^2\right)^{\frac{1}{2}}$ by the Japanese bracket      $\langle g\rangle$ for $g \in \mathbb{H}^n$.}
Then, under some additional assumptions on the initial data, we have the following blow-up result.
\begin{theorem}\label{blow-up} Let $\gamma \in\left(0, \frac{Q}{2}\right)$ and  let the exponent $p$ satisfy  $1<p<p_{\text {crit }}(Q, \gamma)$. We   also assume  that the non-negative initial data $\left(u_0, u_1\right) \in \dot{H}_\mathcal{L}^{-\gamma} \times \dot{H}_\mathcal{L}^{-\gamma}$ satisfies 
	\begin{align}\label{eq32}
		u_0(g)+u_1(g) &\geq C_0 \langle g \rangle^{-Q\left(\frac{1}{2}+\frac{\gamma}{Q}\right)}(\log (e+|g|))^{-1}, \quad g\in \HH,
	\end{align}
	where $C_0>0$ is a fixed constant. Then, there is no global (in time) weak solution to the Cauchy problem (\ref{eq0010}).
	Moreover, the lifespan $T_{w,\varepsilon}$ of local (in time) weak solutions to the Cauchy problem  (\ref{eq0010}) satisfies the following upper bound estimate
	$$
{   T_{w, \varepsilon} \leq C \varepsilon^{-\left(\frac{1}{p-1}-\left(\frac{Q}{4}+\frac{\gamma}{2}\right)\right)^{-1}},}
	$$
 where  $C$ is a positive constant independent of $\varepsilon$.
	
\end{theorem}


From  Theorems \ref{well-posed} and Theorem \ref{blow-up}, we can conclude that the critical exponent for the semilinear damped wave equation (\ref{eq0010})   is 	$p_{\text {crit }}(Q, \gamma):=1+\frac{4}{Q+2 \gamma}$ (see (\ref{Critical-exponent})) when initial data are additionally taken from the negative order Sobolev space  $ \dot{H}_\mathcal{L}^{-\gamma}, \gamma>0$. It gives us a new way to look at the critical exponent for the semilinear damped wave equation in Sobolev space with a negative order.  
 {   Indeed, we can interpret the exponent $p_{\text {crit }}(Q, \gamma)$  as a modification of the exponent $p_{\text {Fujita }}\left(\frac{Q}{2}\right)=1+\frac{4}{Q}$ (which would correspond to the critical exponent with only $L^2$ regularity for the Cauchy data) by working with even weaker regularity ($H_{\mathcal{L}}^{-\gamma}$ regularity), the critical exponent becomes even smaller $p_{\text {crit }}(Q, \gamma)=1+\frac{4}{Q+2 \gamma}$.
 	 Since,  we are working with less regularity for the data, the decay rates in the estimates for the solutions of the homogeneous problem improve and the Banach fixed point theorem holds for a larger range of $p$.}

The behavior of the  {  Cauchy problem} at the critical case $p=p_{\text {crit }}(Q, \gamma) $  remains uncertain, as it is unclear whether  a  global (in time)  small data Sobolev solutions exist or the weak solutions will blow-up.  

In the next remark, particularly,	we consider  $s=1$ for  Theorem \ref{well-posed} and \ref{blow-up} to give a rough idea about  the critical exponent.
\begin{rem}
	Due to the  applications of   Gagliardo-Nirenberg inequality, we have   the technical restriction on the exponent   $1<p \leq \frac{Q}{Q-2}$. Consequently, the ranges of exponent $p$ for global-in-time existence and blow-up are as follows.
	\begin{itemize}
		\item 	When $Q=4,6$ 
		\begin{itemize}
			\item blow-up of weak solutions if $1<p<p_{\text {crit }}(Q, \gamma)$, 
			\item and global-in-time existence of Sobolev solutions if $$p_{\text {crit }}(Q, \gamma)<p \leq \frac{Q}{Q-2}, \qquad   \text{for}~0<\gamma \leq \tilde{\gamma},$$ as well as $$1+\frac{2 \gamma}{Q} \leq p \leq \frac{Q}{Q-2}, \qquad \text{for}~\tilde{\gamma}<\gamma<\frac{Q}{2}.$$
		\end{itemize} 
		\item When $Q\geq 8$, from Theorem \ref{well-posed}, the set of admissible exponents  $p$   is completely empty.
	\end{itemize}
\end{rem}
\begin{figure}[b]
	\centering
	\begin{tikzpicture}[>=latex,xscale=2.35,yscale=1.75,scale=0.82]
		\draw[->] (5.8,0) -- (9,0) node[below] {$\gamma$};
		\draw[->] (6,-0.2) -- (6,3.4) node[left] {$p$};
		\node[left] at (6,-0.2) {{$0$}};
		\node[left] at (6,2.7) {{$\frac{Q}{Q-2}$}};
		\node[below] at (7,0) {{${\tilde{\gamma}}$}};
		\node[below] at (7.9,0) {{$\frac{Q}{2}$}};
		\node[left] at (6,1) {{$1$}};
		\node[left, color=black] at (9.138,1.8) {{$\longleftarrow$ $p=1+\frac{2\gamma}{Q}$}};
		\node[right, color=black] at (4.44,2) {{ $p=p_{\mathrm{crit}}(Q,\gamma)$$\longrightarrow$}};
		\draw[dashed, color=black]  (6, 1)--(8, 1);
		\draw[dashed, color=black]  (8, 0)--(8, 3.4);
		\draw[dashed, color=black]  (6, 2.7)--(8, 2.7);
		\draw[dashed, color=black]  (7, 0)--(7, 1.5);
		\draw[color=black] plot[smooth, tension=.7] coordinates {(6,2.4) (6.3,1.9) (7,1.5)};
		\draw[dashed, color=black] plot[smooth, tension=.7] coordinates {(6.2,3.2) (6.4,2.3) (7,1.5)};
		\node[below] at (6.3,3.2) {{$\nearrow$}};
		\node[below] at (7.5,3.2) {{$\nwarrow$}};
		\draw[color=black] plot[smooth, tension=.7] coordinates { (6,1) (8,2)};
		\draw[dashed, color=black] plot[smooth, tension=.7] coordinates { (7,1.5) (7.7,3.3)};
		\draw[dashed, color=black] plot[smooth, tension=.7] coordinates { (7,1.5) (7.5,1.4) (8,1.3)};
		\fill [pattern=north west lines, pattern color=red]  (6,2.4) -- (6.3,1.9) -- (7,1.5) -- (7.5,1.4) --(8,1.3) -- (8,1) -- (6, 1);
		\fill [pattern=north east lines, pattern color=green]  (6,2.4) -- (6.3,1.9) -- (7,1.5) -- (8,2) -- (8,2.7) -- (6,2.7);
	\end{tikzpicture}
	
	\begin{center}

		\begin{tikzpicture}[>=latex,xscale=1.35,yscale=1.25,scale=0.82]	\node[left] at (9.2,-1) {{For  $Q=4,6$}};	\end{tikzpicture}\\
		\begin{tikzpicture}[>=latex,xscale=1.35,yscale=1.25,scale=0.82]	\node[left] at (4.81,4.1) {{\color{green} $\square$ Range for global existence of solution}};\fill [pattern=north east lines, pattern color=green]  (-1.51,4) -- (-1.51, 4.3) -- (-1.81,4.3) -- (-1.81, 4);	\end{tikzpicture}\\
		\begin{tikzpicture}[>=latex,xscale=1.35,yscale=1.25,scale=0.82]	\node[left] at (4.8,4.1) {{   $\square$ Range for blow-up of solution}}; \fill [pattern=north west lines, pattern color=red]   (-0.34,4) -- (-0.34, 4.3) -- (-0.64,4.3) -- (-0.64, 4); \end{tikzpicture}
	\end{center}
	\caption{Description of the critical exponent in the $(\gamma, p)$ plane}
	\label{imgg2}
\end{figure}
For a detailed analysis of the critical exponent which depends on the parameter $\gamma$ and $Q$, we  can   describe   it by the $(\gamma, p) $ plane in Figure \ref{imgg2}. 	In  Figure \ref{imgg2},  with  an  increase of the homogeneous dimension $Q$ of $\mathbb{H}^n$, the curve $p=p_{\text {crit }}(Q, \gamma)$ and the segment $p=1+\frac{2 \gamma}{Q}$ will move following the direction of   $\nearrow$ and $\nwarrow$ lines arrows, respectively. 

The formulations of Theorem \ref{well-posed} and Theorem \ref{blow-up} for $s=1$ also give the critical index for regularity of initial data that belongs additionally to $ \dot{H}_\mathcal{L}^{-\gamma}$ with $0<\gamma<\tilde{\gamma} $. Specifically, for $1+\frac{2}{Q}<p<1+\frac{4}{Q}$, the critical index is given by $\gamma_{\text {crit }}(p, Q):=\frac{2}{p-1}-\frac{Q}{2}$. Regarding the initial data that also  additionally belongs to $ \dot{H}_\mathcal{L}^{-\gamma}$ with $\gamma >0$,
\begin{itemize}
	\item  then local (in-time) weak solutions in general blow up in finite time for $0<\gamma<\gamma_{\text {crit }}(p, Q)$;
	\item then the global (in-time) small data Sobolev solutions exists uniquely for $\gamma_{\text {crit }}(p, Q)<\gamma<\min  \{\tilde{\gamma}, \frac{Q}{2} \}$.
\end{itemize}    

From Theorem \ref{blow-up},    for  $1 < p < p_{\text {crit }}(Q, \gamma)$, we know that   the   non-trivial local (in time) weak solution   blow up in finite time   and we have the  following  upper bound estimates for the lifespan
$$T_\varepsilon\leq 	T_{w,\varepsilon} \leq C \varepsilon^{-\left(\frac{1}{p-1}-\left(\frac{Q}{4}+\frac{\gamma}{2}\right)\right)^{-1}}.$$
Now, it is important to investigate the lower bound estimates for the lifespan.  

To thoroughly examine the lifespan, it is essential to consider the precise definition of mild solutions to the Cauchy problem (\ref{eq0010}) on $[0,T)$ with $T > 0$ for $u \in  C([0,T),H_\mathcal{L}^1)$. Let   $T_{m,\varepsilon}$ be the lifespan of a mild solution $u $. Then, we have the following result regarding the lower bound for the lifespan.
\begin{theorem}\label{lower bound}  
	Let $\gamma \in (0, \tilde{\gamma} ) $ and let the  exponent $p$ satisfy $1<p<p_{\text {crit }}(Q, \gamma)$ such that
	\begin{align}\label{eq35}
		1+\frac{2 \gamma}{Q} \leq p \leq \frac{Q}{Q-2 }.
	\end{align}
	We also   assume that  $\left(u_0, u_1\right) \in \mathcal{A}_1^{\mathcal{L}}$. Then, there exists a constant $\varepsilon_0$ such that for every $\varepsilon \in\left(0, \varepsilon_0 \right]$, the lifespan $T_{m, \varepsilon}$ of  mild solutions  $u$ to the Cauchy problem (\ref{eq0010}) satisfies the following lower bound condition:
	$$
	T_{m,\varepsilon} \geq  D \varepsilon^{-\left(\frac{1}{p-1}-\left(\frac{Q}{4}+\frac{\gamma}{2}\right)\right)^{-1}},
	$$
	where $D$ is a positive constant independent of $\varepsilon$, but  may  depends on $p, Q, \gamma$ as well as $\left\|\left(u_0, u_1\right)\right\|_{\mathcal{A}_1^{\mathcal{L}}}$.
\end{theorem} 

Since $u$ in Theorem \ref{well-posed} is a mild solution to (\ref{eq0010}),   this mild solution is also a weak solution to (\ref{eq0010})  {  (according to a density argument)} with timespan $ T_{m,\varepsilon}\leq  T_\varepsilon$.  Once more drawing on Theorems \ref{blow-up} and \ref{lower bound}, for  $1<p<p_{\text {crit }}(Q, \gamma)$, we claim a sharp estimate for the lifespan $T_\varepsilon$  as stated below
$$
T_\varepsilon \simeq C \varepsilon^{-\left(\frac{1}{p-1}-\left(\frac{Q}{4}+\frac{\gamma}{2}\right)\right)^{-1}},
$$
for $\gamma \in (0, \min  \{2, \frac{Q}{2} \} )$. Moreover, the lifespan estimate mentioned above agrees with the one for additional $L^1$-regular data when particularly we consider $\gamma=\frac{Q}{2}$.

\begin{rem}
	One interesting observation is to discuss the acceptable ranges for $p$ so that we can get sharp lifespan estimates.  This can be summed up as follows:
	\begin{itemize}
		\item	when $Q=4$, for $1+\frac{2 \gamma}{Q} \leq  p<p_{\text {crit }}(Q, \gamma)$ and $0<\gamma \leq  \frac{-Q+\sqrt{Q^2+16 Q}}{4}$,  we can achieve sharp lifespan estimates;\vspace{0.1cm}
		\item 	When $Q=6$, for $1+\frac{2 \gamma}{Q} \leq p<p_{\text {crit }}(Q, \gamma)$ with $p \leq \frac{Q}{Q-2}$ and $0<\gamma \leq  \frac{-Q+\sqrt{Q^2+16 Q}}{4}$, sharp lifespan estimates can be achieved;\vspace{0.1cm}
		\item 	when $Q \geq 8$, for $1+\frac{2 \gamma}{Q} \leq  p \leq \frac{Q}{Q-2}$ and $0<\gamma \leq \frac{Q}{Q-2}$, we can achieve sharp lifespan estimates.
	\end{itemize}

\end{rem} 

We conclude the introduction with a brief outline of the organization of the paper. 	
In Section \ref{sec1}, we discuss the main results and their explanations related to  the critical exponent  $p_{\text{crit}}(Q, \gamma)$  and    sharp lifespan estimates for weak solutions  to Cauchy problem (\ref{eq0010}) when  initial data taken additionally from $ \dot {H}_\mathcal{L}^{-\gamma}$. Section \ref{sec2} is devoted to recalling some basics of the Fourier analysis on the Heisenberg  group $\mathbb{H}^n$ to make the paper self-contained.  Using  the  Fourier analysis on the Heisenberg group $\HH$,  we derive $\dot H^s_{\mathcal{L}}(\HH)$-decay estimates  for the solution to  the linear  damped wave equation (\ref{eq0010}) with vanishing right-hand side in Section \ref{sec3}. Using the $\dot H^s_{\mathcal{L}}(\HH)$-decay estimates  for the solution to a linear   damped wave equation, we demonstrate global-in-time well-posedness for the semilinear Cauchy problem (\ref{eq0010}) if $p>p_{\text {crit }}(Q, \gamma)$  with the help of Banach's fixed point argument in Section \ref{sec4}.  In order to estimate  the nonlinear term in  $\dot{H}_\mathcal{L}^{-\gamma}(\HH)$, we use   Sobolev inequality and the   Gagliardo-Nirenberg inequality on the Heisenberg group $\HH$.  In Section \ref{sec5}, by employing the test function method, we  conclude the optimality by showing the blow-up of weak solutions even for small data for the range  $1<p<p_{\text {crit }}(Q, \gamma)$.  As a byproduct, we also obtain upper bound estimates for the lifespan.  We conclude our paper by deriving the sharp lower bound estimates for the lifespan of mild solutions in Section \ref{sec6} by using the method of contradiction.

\section{Preliminaries: Analysis on  the Heisenberg group $\mathbb{H}^n$} \label{sec2}
In this section, we recall some basics of the Fourier analysis on the Heisenberg    groups $\mathbb{H}^n$ to make the manuscript self-contained. A complete account of the representation theory on $\mathbb{H}^n$ can be found in \cite{Vla, thanga, 30, Fischer,palmieri,Poho}. However, we mainly adopt the notation and terminology given in \cite{Fischer} for convenience.
We commence this section by establishing the notations that will be employed consistently throughout the paper.
\subsection{Notations} 
Throughout the article,  we use the following notations:  

\begin{itemize}
	\item $f \lesssim g:$\,\, There exists a positive constant $C$ (whose value may change from line to line in this manuscript) such that $f \leq C g.$  
	\item  $f \simeq g$: Means that $f \lesssim g$ and $g \lesssim f$. 
	\item $\mathbb{H}^n:$  The Heisenberg  group.
	\item $Q$:  The homogeneous dimension of $\HH$.
	\item $dg:$ The Haar measure on the Heisenberg  group $\HH.$
	\item $\mathcal{L}:$ The sub-Laplacian  on $\HH.$
	\item  ${H}_\mathcal{L}^{-\gamma}$: The subelliptic Sobolev spaces of negative order  with $\gamma>0$ on $\HH$.
\end{itemize}

\subsection{The Hermite operator} In this  subsection, we recall some  definitions and properties of Hermite functions which we will use frequently in order to study Schr\"odinger representations and sub-Laplacian on  the Heisenberg group $\HH$.  We start with the definition of Hermite polynomials on $\mathbb{R}.$

Let $H_k$ denote the Hermite polynomial on $\mathbb{R}$, defined by
$$H_k(x)=(-1)^k \frac{d^k}{dx^k}(e^{-x^2} )e^{x^2}, \quad k=0, 1, 2, \cdots   ,$$\vspace{.30mm}
and $h_k$ denote the normalized Hermite functions on $\mathbb{R}$ defined by
$$h_k(x)=(2^k\sqrt{\pi} k!)^{-\frac{1}{2}} H_k(x)e^{-\frac{1}{2}x^2}, \quad k=0, 1, 2, \cdots,$$\vspace{.30mm}
The Hermite functions $\{h_k \}$ are the eigenfunctions of the Hermite operator (or the one-dimensional harmonic oscillator) $H=-\frac{d^2}{dx^2}+x^2$ with eigenvalues $2k+1,  k=0, 1, 2, \cdots$. These functions form an orthonormal basis for $L^2(\mathbb{R})$. The higher dimensional Hermite functions denoted by $e_{k}$ are then obtained by taking tensor products of one dimensional Hermite functions. Thus for any multi-index $k=(k_1, \cdots, k_n) \in \mathbb{N}_0^n$ and $x =(x_1, \cdots, x_n)\in \mathbb{R}^n$, we define
$e_{k}(x)=\prod_{j=1}^{n}h_{k_j}(x_j).$
The family $\{e_{k}\}_{k\in \mathbb{N}^n_0}$ is then   an orthonormal basis for $L^2(\mathbb{R}^n)$. They are eigenfunctions of the Hermite operator $\mathrm{H}=-\Delta+|x|^2$, namely, we have an ordered set of  {  natural  numbers} $\{\mu_k\}_{k \in \mathbb{N}_0^n}$ such that 
$$\mathrm{H} e_{k}(x)=\mu_ke_{k}(x), \quad x\in \mathbb{R}^n,$$
for all $k\in \mathbb{N}_0^n$ and $x\in \mathbb{R}^n$. More precisely, $\mathrm{H}$  has eigenvalues 
$$\mu_k=\sum_{j=1}^{n}(2k_j+1)=2|k|+n,$$
corresponding to the eigenfunction $e_{k}$ for $k\in \mathbb{N}_0^n.$

{   Given $f \in L^{2}(\mathbb{R}^{n})$, we have the Hermite expansion
$$f=\sum_{k \in \mathbb{N}_0^{n}}\left(f, e_{k}\right) e_{k}=\sum_{m=0}^\infty \sum_{|k|=m}\left(f, e_{k}\right) e_{k}= \sum_{m=0}^\infty P_mf,$$ where $P_{m}$ denotes the orthogonal projection of $L^{2}(\mathbb{R}^{n})$ onto the eigenspace spanned by $\{e_{k}:|k|=m\}.$  
Then the spectral decomposition of $H$ on $\mathbb{R}^n$ is given by
$$
Hf=\sum_{m=0}^{+\infty}(2 m+n) P_mf,  
$$
  Since 0 is not in the spectrum of $H$,  for any $s \in \mathbb{R}$, we can define the fractional powers $H^s$  by means of the spectral theorem, namely
\begin{align}\label{Hermite power}
    H^s f=\sum_{m=0}^{\infty}(2 m+n)^s P_m f.
\end{align}
 
}

\subsection{The Heisenberg group}
One of the simplest examples of a non-commutative and non-compact group is the famous Heisenberg group $\mathbb{H}^n$.  The theory of the Heisenberg group   plays a crucial  role in several branches of mathematics and physics. The Heisenberg group $\mathbb{H}^n$  is a nilpotent Lie group whose underlying manifold is $ \mathbb{R}^{2n+1} $ and the group operation is defined by
$$(x, y, t)\circ (x', y', t')=(x+x', y+y', t+t'+ \frac{1}{2}(xy'-x'y)),$$
where $(x, y, t)$, $ (x', y', t')$ are in $\mathbb{R}^n \times \mathbb{R}^n \times \mathbb{R}$ and  $xy'$ denotes the standard scalar product in $\mathbb{R}^n$. Moreover, $\mathbb{H}^n$ is a unimodular Lie group on which the left-invariant Haar measure $dg$ is the usual Lebesgue measure $\, dx \, dy \,dt.$ 

The canonical basis for the Lie algebra $\mathfrak{h}_n$ of $\mathbb{H}^n$ is given by the left-invariant vector fields:
\begin{align}\label{CH00VF}
	X_j&=\partial _{x_j}-\frac{y_j}{2} \partial _{t}, &Y_j=\partial _{y_j}+\frac{x_j}{2} \partial _{t}, \quad j=1, 2, \dotsc   n,   ~ \mathrm{and} ~ T=\partial _{t},
\end{align} \vspace{.30mm}
which 	satisfy the commutator relations
$[X_{i}, Y_{j}]=\delta_{ij}T, \quad  \text{for} ~i, j=1, 2, \dotsc n.$

Moreover, the canonical basis for   $\mathfrak{h}_n$  admits the  decomposition  $\mathfrak{h}_n=V\oplus W,$ where $V=\operatorname{span}\{X_j, Y_j\}_{j=1}^n$ and $W=\operatorname{span} \{T\}$.   Thus, the Heisenberg group $\HH$ is a step 2 stratified Lie group and its  homogeneous dimension is $Q:=2n +2$.
The sublaplacian $\mathcal{L}$ on $\HH$ is   defined  as
\begin{align*}
	\mathcal{L} &=\sum_{j=1}^{n}(X_j^2+Y_j^2)=\sum_{j=1}^{n}\bigg(\bigg(\partial _{x_j}-\frac{y_j}{2} \partial _{t}\bigg)^2+\bigg(\partial _{y_j}+\frac{x_j}{2} \partial _{t}\bigg)^2\bigg)\\& =\Delta_{\mathbb{R}^{2n}}+\frac{1}{4}(|x|^2+|y|^2) \partial_t^2+\sum_{j=1}^n\left(x_j \partial_{y_j t}^2-y_j \partial_{x_j t}^2\right),
\end{align*}
where $\Delta_{\mathbb{R}^{2n}}$ is the standard Laplacian on $\mathbb{R}^{2n}.$

\subsection{Fourier analysis on  the Heisenberg group $\mathbb{H}^n$} We start this subsection by  recalling the definition of the operator-valued group Fourier transform on $\mathbb{H}^n$.  By Stone-von Neumann theorem, the only infinite-dimensional unitary irreducible representations
(up to unitary equivalence) are given by $\pi_{\lambda}$, $\lambda$ in $\mathbb{R}^*$, where the mapping  $\pi_{\lambda}$  is a strongly continuous unitary representation defined by 
$$\pi_{\lambda}(g)f(u)=e^{i \lambda(t+\frac{1}{2}xy)} e^{i \sqrt{\lambda}yu} {  f(u+\sqrt{|\lambda|}x)}, \quad    g=(x, y, t)\in \mathbb{H}^n,$$ for all  $ f  \in L^2(\mathbb{R}^n).$ We use the convention $$\sqrt{\lambda}:={\rm sgn}(\lambda)\sqrt{|\lambda|}
= \begin{cases}  \sqrt{\lambda},&\lambda>0,\\ -\sqrt{|\lambda|}, &\lambda<0.\end{cases}
$$
For each ${\lambda} \in \mathbb{R}^*$, the group Fourier transform of $f\in L^1(\mathbb{H}^n)$ is a bounded linear operator   on  $L^2(\mathbb{R}^n)$ defined by
\begin{align*}
	\widehat{f}(\lambda) \equiv \pi_{\lambda}(f)=\int_{\mathbb{H}^n}f(g)\pi_{\lambda}^*(g)\, dg.
\end{align*}
Let   $B(L^2(\mathbb{R}^n))$  be the set of all bounded operators on $L^2(\mathbb{R}^n)$. As the Schr\"odinger representations are unitary,  for any  {  $\lambda\in \mathbb{R}^*$}, we have  
$$\|\widehat{f}\|_{B(L^2(\mathbb{R}^n))}\leq \|f\|_{L^1(\HH)}.$$
If $f \in L^2(\mathbb{H}^n)$, then $\widehat{f}(\lambda)$ is a Hilbert-Schmidt operator on $L^2(\mathbb{R}^n)$ and satisfies the following Plancherel formula $$\|{f}\|_{L^2(\mathbb{H}^n)}^2=\int_{\mathbb{R}^*}\|\widehat{f}(\lambda)\|_{S_2}^{2} \, d\mu( \lambda),$$
where $\| . \|_{S_2}$ stands for  the norm in the Hilbert space $S_2$,   the set of all Hilbert-Schmidt operators  on  $L^2(\mathbb{R}^n)$ and $d\mu(\lambda)=c_n {|\lambda|}^n \, d\lambda$ with $c_n$ being  a positive constant.

Therefore, with the help of the orthonormal basis $\left\{e_k\right\}_{k \in \mathbb{N}_0^n}$ for  $L^2\left(\mathbb{R}^n\right)$ and  the definition of the Hilbert-Schmidt norm, we have 
$$
\begin{aligned}
	\|\widehat{f}(\lambda)\|_{S_2}^2 & \doteq \operatorname{Tr}\left[  (\pi_\lambda(f) )^* \pi_\lambda(f)\right] \\
	& =\sum_{k \in \mathbb{N}_0^n} \|\widehat{f}(\lambda) e_k \|_{L^2\left(\mathbb{R}^n\right)}^2=\sum_{k, \ell \in \mathbb{N}^n_0} | (\widehat{f}(\lambda) e_k, e_{\ell} )_{L^2\left(\mathbb{R}^n\right)} |^2.
\end{aligned}
$$
The above expression allows us to write the Plancherel formula  in the following way:
$$\|{f}\|_{L^2(\mathbb{H}^n)}^2=\int_{\mathbb{R}^*}\sum_{k, \ell \in \mathbb{N}^n_0} | (\widehat{f}(\lambda) e_k, e_{\ell} )_{L^2\left(\mathbb{R}^n\right)} |^2 \, d\mu( \lambda).$$
For  $f \in \mathcal{S}(\mathbb{H}^n)$, the space of all Schwartz class functions on $\HH,$ the Fourier inversion formula  takes the form
\begin{align*}
	f(g) &:=  \int_{\mathbb{R}^*}\operatorname{Tr}[\pi_{\lambda}(g)\widehat{f}(\lambda)]\,d\mu(\lambda), \quad   g\in \mathbb{H}^n,
\end{align*}
where $\operatorname{Tr(A)}$ denotes the trace  of the operator $A$.

Furthermore, the action of the infinitesimal representation $d\pi_\lambda$ of $\pi_\lambda$ on the generators of the first layer of the Lie algebra $\mathfrak{h}_n$ is given by
$$
\begin{aligned}
	\;d \pi_\lambda\left(X_j\right) & =\sqrt{|\lambda|} \partial_{x_j} & \text { for } j=1, \cdots, n, \\
	\mathrm{~d} \pi_\lambda\left(Y_j\right) & =i \operatorname{sign}(\lambda) \sqrt{|\lambda|} x_j & \text { for } j=1, \cdots, n .
\end{aligned}
$$
Since the action of $\;d \pi_\lambda$ can be extended to the universal enveloping algebra of $\mathfrak{h}_n$, combining  the above two expressions,  we obtain
$$
\;d \pi_\lambda\left(\mathcal{L}\right)=\;d \pi_\lambda\bigg(\sum_{j=1}^n(X_j^2+Y_j^2)\bigg)=|\lambda| \sum_{j=1}^n\big(\partial_{x_j}^2-x_j^2\big)=-|\lambda| \mathrm{H},
$$
where $\mathrm{H} =-\Delta+|x|^2$ is the Hermite operator on $\mathbb{R}^n$. Thus the operator valued symbol $\sigma_\mathcal{L}(\lambda)$ of $\mathcal{L}$ acting on $L^2(\mathbb{R}^n)$  takes the form
$\sigma_\mathcal{L}(\lambda)=-|\lambda| \mathrm{H}.$ {  Furthermore, for $s \in \mathbb{R}$, using  the functional calculus, the     symbol of $(-\mathcal{L})^s $ is 
$|\lambda|^s \mathrm{H}^s,$ where the notion of $\mathrm{H}^s$ is defined in (\ref{Hermite power}).
}

The Sobolev spaces $H_{\mathcal{L}}^s, s \in \mathbb{R}$, associated to the sublaplacian $\mathcal{L}$, are defined as
$$
H_{\mathcal{L}}^s\left(\mathbb{H}^n\right):=\left\{f \in \mathcal{D}^{\prime}\left(\mathbb{H}^n\right):(I-\mathcal{L})^{s / 2} f \in L^2\left(\mathbb{H}^n\right)\right\},
$$
with the norm $$\|f\|_{H_{\mathcal{L}}^s\left(\mathbb{H}^n\right)}:=\left\|(I-\mathcal{L})^{s / 2} f\right\|_{L^2\left(\mathbb{H}^n\right)}.$$ 
Similarly,  we denote by $ \dot{H}_{\mathcal{L}}^{ s}(\mathbb{H}^n),$ the  homogeneous Sobolev  defined as the space of all $f\in \mathcal{D}'(\HH)$ such that $(-\mathcal{L})^{{s}/{2}}f\in L^2(\HH)$. More generally, 	we define  $ \dot{H}_{\mathcal{L}}^{p, s}(\mathbb{H}^n)$ as the  homogeneous Sobolev  defined as the space of all  $f\in \mathcal{D}'(\HH)$ such that $(-\mathcal{L})^{{s}/{2}}f\in L^p(\HH)$ for $s <\frac{Q}{p}$. 	Then we recall the following important inequalities, see e.g. \cite{30,Fischer,DKR23, GKR, GKR2} for  a more general graded Lie group framework. However, we will state those in the Heisenberg group setting.

\begin{theorem}[Hardy-Littlewood-Sobolev inequality]\label{eq177} Let  $\HH$ be the Heisenberg group with the homogeneous dimension $Q:=2n+2$.  Let $a\geq 0$ and $1<p\leq q<\infty$ be such that 	$$	\frac{a}{Q}= \frac{1}{p}-\frac{1}{q}.$$ Then we have   the following inequality $$ \|f\|_{\dot{H}_{\mathcal{L}}^{q, -a}(\HH)} \lesssim \|f\|_{L^p(\HH)}. 	$$
\end{theorem}

We have the Gagliardo-Nirenberg inequality  on  $\mathbb{H}^n$ as follows:
\begin{theorem}[Gagliardo-Nirenberg inequality]\label{eq16}
	Let  $Q $ be the homogeneous dimension on the Heisenberg group $\HH$.   Assume that
	$$
	s\in(0,1], 1<r<\frac{Q}{s},\text { and }~  2 \leq q \leq \frac{rQ}{Q-sr} 		.$$
	Then we have the following Gagliardo-Nirenberg type inequality,
	$$
	\|u\|_{L^q(\mathbb{H}^n)} \lesssim\|u\|_{\dot{H}_{\mathcal{L}}^{r,s}(\mathbb{H}^n)}^\theta\|u\|_{L^2(\mathbb{H}^n)}^{1-\theta},
	$$
	for $\theta=\left(\frac{1}{2}-\frac{1}{q}\right)/{\left(\frac{s}{Q}+\frac{1}{2}-\frac{1}{r}\right)}\in[0,1]$, provided $\frac{s}{Q}+\frac{1}{2}\neq \frac{1}{r}$.
	
\end{theorem}


\section{Linear damped wave equations: $\dot{H}_{\mathcal{L}}^s(\mathbb{H}^n)$-norm  estimates}\label{sec3}
In this section,  as a preliminary step in preparing to investigate the local and global well-posedness of the nonlinear Cauchy problem (\ref{eq0010}), we examine its associated linear counterpart, which involves a vanishing right-hand side. Specifically, our attention is directed towards the decay properties of solutions, in which our proofs are slightly different from the known results. Let us consider the Cauchy problem
\begin{align}\label{eq001}
	\begin{cases}
		\partial^2_tu-\mathcal{L}u+	\partial_tu =0, & g\in \mathbb{H}^n,t>0,\\
		u(0,g)=  u_0(g),  & g\in \mathbb{H}^n,\\ \partial_tu(0, g)=  u_1(g), & g\in \mathbb{H}^n,
	\end{cases}
\end{align}
where     $u_{0}(g)$ and $u_{1}(g)$ are the initial data additionally belonging to subelliptic Sobolev space $\dot H_{\mathcal{L}}^{-\gamma}(\HH)$ of negative order.  We will derive $\dot H_{\mathcal{L}}^s(\mathbb{H}^n)$-norm  estimates  for  the solution $u(t, \cdot)$   to the homogeneous problem     (\ref{eq001}). We make use of the group Fourier transform on the Heisenberg group $\mathbb{H}^n$, specifically with respect to the spatial variable $g$, and combine it with the Plancherel identity to estimate the $\dot H_{\mathcal{L}}^s(\mathbb{H}^n)$-norm.    {   We refer to  \cite{Palmieri 2020}, where   a similar approach has been carried out for the $L^2$-estimates of the solution to the damped wave equation (\ref{eq001}) on the Heisenberg group $\mathbb{H}^n$.}

Given that the group Fourier transform of a function $f \in L^2\left(\mathbb{H}^n\right)$  is no longer a function but a family of bounded linear operators $\{\widehat{f}(\lambda)\}_{\lambda \in \mathbb{R}^*}$ on $L^2\left(\mathbb{R}^n\right)$, the proofs are more involved and do not follow as in the Euclidean setup. We overcome this obstacle by using a trick to project these operators by using the orthonormal basis $\left\{e_k\right\}_{k \in \mathbb{N}^n}$ of $L^2(\mathbb{R}^n)$, namely, by working with the Fourier coefficients  
$
\left\{\left(\widehat{f}(\lambda) e_k, e_{\ell}\right)_{L^2\left(\mathbb{R}^n\right)}\right\}_{k, \ell \in \mathbb{N}^n}
$
for the Fourier transform $\widehat{f}(\lambda)$ for each $\lambda \in \mathbb{R}^*.$

Invoking the group Fourier transform with respect to $g$ on   (\ref{eq001}),  we get a   Cauchy problem related to a parameter-dependent functional differential equation for $ \widehat{u}(t,\lambda),$ namely, 
\begin{align}\label{eq6661}
	\begin{cases}
		\partial^2_t\widehat{u}(t,\lambda)-	\sigma_{\mathcal{L}}(\lambda) \widehat{u}(t,\lambda) +\partial_t \widehat{u}(t,\lambda)=0,& \lambda \in\mathbb{R}^*,~t>0,\\ \widehat{u}(0,\lambda)=\widehat{u}_0(\lambda), &\lambda \in\mathbb{R}^*,\\ \partial_t\widehat{u}(0,\lambda)=\widehat{u}_1(\lambda), &\lambda \in\mathbb{R}^*,
	\end{cases} 
\end{align}
where  $\sigma_{\mathcal{L}}(\lambda)$ is the symbol of the sub-Laplacian $\mathcal{L}$ on $\HH$.  In fact, we know that   $\sigma_{\mathcal{L}}(\lambda)=-|\lambda| \mathrm{H}$.   For any $ k, \ell \in \mathbb{N}^n$, let us introduce the notation
$$
  {   \widehat{u}(t, \lambda)_{k, \ell} \doteq\left(\widehat{u}(t, \lambda) e_\ell, e_{k}\right)_{L^2\left(\mathbb{R}^n\right)},}
$$
where $\left\{e_k\right\}_{k \in \mathbb{N}^n}$ is the system of Hermite functions forming an orthonormal basis of $L^2(\mathbb{R}^n)$. Since $\mathrm{H} e_k=\mu_k e_k$, $\widehat{u}(t, \lambda)_{k, \ell}$ solves an ordinary differential equation with respect to the variable $t$ depending on parameters $\lambda \in \mathbb{R}^*$ and $k, \ell \in \mathbb{N}^n,$	
\begin{align}\label{eqq7}
	\begin{cases}
		\partial^2_t\widehat{u}(t,\lambda)_{kl}+	\partial_t\widehat{u}(t,\lambda)_{kl}+ \beta_{k, \lambda}^{2 }  \widehat{u}(t,\lambda)_{kl}= 0,& \lambda \in\mathbb{R}^*,~t>0,\\ \widehat{u}(0,\lambda)_{kl}=\widehat{u}_0(\lambda)_{kl}, &\lambda \in\mathbb{R}^*,\\ \partial_t\widehat{u}(0,\lambda)_{kl}=\widehat{u}_1(\lambda)_{kl}, &\lambda \in\mathbb{R}^*,
	\end{cases}
\end{align}
where $\beta_{k, \lambda}^{2 } =|\lambda|\mu_k.$	Then, the characteristic equation of (\ref{eqq7}) is given by
\[m^2+	m+\beta_{k, \lambda}^{2 } =0,\]
and so the characteristic roots $m_1$ and $m_2$   are     $\frac{-1- \sqrt{1-4\beta_{k, \lambda}^{2 }}}{2}  $ and $\frac{-1+ \sqrt{1-4\beta_{k, \lambda}^{2 }}}{2},$ respectively. 

Observe that,  for    $|\beta_{k, \lambda}|  \ll 1$, we get
\begin{align*}
	&  m_1=\frac{-1- \sqrt{1-4\beta_{k, \lambda}^{2 }}}{2}
	=\frac{-1- (1-4{\beta_{k, \lambda}}^2 )^{\frac{1}{2}}}{2}=-1+\mathcal{O}\left(\beta_{k, \lambda}^{2 }\right),\\
	&     m_2= 
	\frac{-1+ (1-4{\beta_{k, \lambda}}^2 )^{\frac{1}{2}}}{2}=-\beta_{k, \lambda}^{2 }+\mathcal{O}\left(\beta_{k, \lambda}^4\right).
\end{align*}
Again,  for $|\beta_{k, \lambda}| \gg 1$, we obtain
\begin{align*}
	m_{1}= \frac{-1- (1-4\beta_{k, \lambda}^2 )^{\frac{1}{2}}}{2}  &=\frac{-1- 2i|\beta_{k, \lambda}| (1-(4\beta_{k, \lambda}^2)^{-1})^{\frac{1}{2}}} {2}  
	\\&=- i|\beta_{k, \lambda}|-\frac{1}{2}+\mathcal{O}\left(|\beta_{k, \lambda}|^{-1}\right),\end{align*}
\begin{align*}
	m_{2}= \frac{-1+(1-4\beta_{k, \lambda}^2 )^{\frac{1}{2}}}{2}  &=\frac{-1+ 2i|\beta_{k, \lambda}| (1-(4\beta_{k, \lambda}^2)^{-1})^{\frac{1}{2}}} {2}  
	\\&= i|\beta_{k, \lambda}|-\frac{1}{2}+\mathcal{O}\left(|\beta_{k, \lambda}|^{-1}\right) .\end{align*}
We collect the above easy observations as  the following relations:
\begin{itemize}
	\item 	$m_1=-1+\mathcal{O}\left(\beta_{k, \lambda}^{2 }\right), m_2=-\beta_{k, \lambda}^{2 }+\mathcal{O}\left(\beta_{k, \lambda}^4\right)$ for $|\beta_{k, \lambda}|<\varepsilon \ll 1$;
	\item $m_{1}= - i|\beta_{k, \lambda}|-\frac{1}{2}+\mathcal{O}\left(|\beta_{k, \lambda}|^{-1}\right)$ and $m_{2}=  i|\beta_{k, \lambda}|-\frac{1}{2}+\mathcal{O}\left(|\beta_{k, \lambda}|^{-1}\right)$ for $|\beta_{k, \lambda}|>N \gg 1$;
	\item    {  $\text{Re}( m_{1})<0$ and $\text{Re}( m_{2})<0$ for $\varepsilon \leq|\beta_{k, \lambda}| \leq N$.}
\end{itemize}
Therefore, the solution to the   homogeneous  problem   (\ref{eqq7}) is given by
\begin{align}\label{Asymptotic exp} 
	\widehat{u}(t,\lambda)_{kl}&= A_0(t, \lambda)_{k\ell} \widehat{u}_0(\lambda)_{kl}+ A_1(t, \lambda)_{k\ell}\widehat{u}_1(\lambda)_{k\ell},
\end{align}
where
\begin{align}\label{shyam} 
	& 	A_0(t, \lambda)_{k\ell}=\frac{m_1{e}^{m_2 t}-m_2{e}^{m_1 t}}{m_1-m_2} \nonumber \\\\&=\begin{cases}
		\frac{\left(-1+\mathcal{O}(\beta_{k, \lambda}^2)\right){e}^{\left(-\beta_{k, \lambda}^{2 }+\mathcal{O} (\beta_{k, \lambda}^{4} )\right)t}-\left(-\beta_{k, \lambda}^{2 }+\mathcal{O} (\beta_{k, \lambda}^{4})\right){e}^{\left(-1 +\mathcal{O} (\beta_{k, \lambda}^{2} )\right)t}}{-1+\mathcal{O}\left(\beta_{k, \lambda}^{2 }\right)} & \text { for }|\beta_{k, \lambda}|<\varepsilon,  \\\\ 	\frac{\left(i|\beta_{k, \lambda}|-\frac{1}{2}+\mathcal{O}\left(|\beta_{k, \lambda}|^{-1}\right)\right)  {e}^{\left(-i|\beta_{k, \lambda}|-\frac{1}{2}+\mathcal{O}\left(|\beta_{k, \lambda}|^{-1}\right)\right) t}}{2 i|\beta_{k, \lambda}|+\mathcal{O}(1)} &  \\\\
		\qquad 	-\frac{\left(-i|\beta_{k, \lambda}|-\frac{1}{2}+\mathcal{O}\left(|\beta_{k, \lambda}|^{-1}\right)\right)  {e}^{\left(i|\beta_{k, \lambda}|-\frac{1}{2}+\mathcal{O}\left(|\beta_{k, \lambda}|^{-1}\right)\right) t}}{2 i|\beta_{k, \lambda}|+\mathcal{O}(1)}
		& \text { for }|\beta_{k, \lambda}|>N .
	\end{cases} 
\end{align}

and
\begin{align} \label{shyam2}
	&\nonumber	A_1(t, \lambda)_{k\ell}=\frac{ {e}^{m_1 t}-{e}^{m_2 t}}{m_1-m_2}\\\\&=\left\{\begin{array}{ll}
		\frac{ {e}^{\left(-1+\mathcal{O}\left(\beta_{k, \lambda}^2\right)\right) t}- {e}^{\left(-\beta_{k, \lambda}^2+\mathcal{O}\left(\beta_{k, \lambda}^4\right)\right) t}}{-1+\mathcal{O}\left(\beta_{k, \lambda}^2\right)} & \text { for }|\beta_{k, \lambda}|<\varepsilon, \\\\
		\frac{ {e}^{\left(i|\beta_{k, \lambda}|-\frac{1}{2}+\mathcal{O}\left(|\beta_{k, \lambda}|^{-1}\right)\right) t}- {e}^{\left(-i|\beta_{k, \lambda}|-\frac{1}{2}+\mathcal{O}\left(|\beta_{k, \lambda}|^{-1}\right)\right) t}}{2 i|\beta_{k, \lambda}|+\mathcal{O}(1)}  & \text { for }|\beta_{k, \lambda}|>N .
	\end{array}\right.
\end{align}
Thus from the above collected observations and asymptotic expression,   we deduce the following  pointwise estimates
\begin{align}\label{K_0estimate}
	\left|A_0(t,\lambda)_{k\ell}\right| \lesssim\left\{\begin{array}{ll}
		|\beta_{k, \lambda}|^2  {e}^{-c t}+ {e}^{-ct\beta_{k, \lambda}^2 } & \text { for }|\beta_{k, \lambda}|<\varepsilon \ll 1, \\\\
		\mathrm{e}^{-c t} & \text { for } \varepsilon \leq|\beta_{k, \lambda}| \leq N, \\\\
		\mathrm{e}^{-c t} & \text { for }|\beta_{k, \lambda}|>N \gg 1,
	\end{array}\right.
\end{align}
and 
\begin{align}\label{K_1estimate}
	\left|A_1(t,\lambda)_{k\ell}\right| \lesssim\left\{\begin{array}{ll}
		{e}^{-c t}+ {e}^{-ct\beta_{k, \lambda}^2 } & \text { for }|\beta_{k, \lambda}|<\varepsilon \ll 1, \\\\
		{e}^{-c t} & \text { for } \varepsilon \leq|\beta_{k, \lambda}| \leq N, \\\\
		|\beta_{k, \lambda}|^{-1}  {e}^{-c t} & \text { for }|\beta_{k, \lambda}|>N \gg 1,
	\end{array}\right.
\end{align}
for some suitable positive constant $c$.

Before  finding  the Sobolev norm of $u(t, \cdot)$, first we notice that, for $ \beta_{k, \lambda}^2 = |\lambda| \mu_k<\varepsilon^2$, we have
\begin{align}\label{decay}
	\left| (|\lambda| \mu_k)^{s+\gamma}     {e}^{-ct|\lambda| \mu_k}\right|\lesssim (1+t)^{-(s+\gamma)}, \qquad \text{for~} s+\gamma \geq 0.
\end{align}

Now by using the Plancherel formula and the fact that $\{e_k\}_{k\in \mathbb{N}^n}$  is an  orthonormal basis of $L^2(\mathbb{R}^n)$, we have 
\begin{align}\label{eq01}\nonumber
	&\|u(t, \cdot )\|_{\dot H^s_{\mathcal{L}}}^2=\int_{\mathbb{R}^*}\|(-\sigma_{\mathcal{L}}(\lambda))^{\frac{s}{2}} \widehat{u}(t, \lambda)\|_{S_2}^2  \,d\mu(\lambda)\\\nonumber
	&\qquad=\int_{\mathbb{R}^*} \sum_{k, \ell\in \mathbb{N}^n} (|\lambda| \mu_k)^{s} |\widehat{u}(t, \lambda)_{k\ell}|^2  \,d\mu(\lambda)\\\nonumber
	&\qquad\lesssim \int_{\mathbb{R}^*} \sum_{k, \ell\in \mathbb{N}^n} (|\lambda| \mu_k)^{s}\left[  |A_0(t, \lambda)_{k\ell}|^2  |\widehat{u}_0(\lambda)_{kl}|^2 + |A_1(t, \lambda)_{k\ell}|^2 |\widehat{u}_1(\lambda)_{k\ell} |^2 \right] d\mu(\lambda)\\
	&\qquad= (I)+(II),
\end{align}
where 
$$(I):=\int_{\mathbb{R}^*} \sum_{k, \ell\in \mathbb{N}^n} (|\lambda| \mu_k)^{s}   |A_0(t, \lambda)_{k\ell}|^2  |\widehat{u}_0(\lambda)_{kl}|^2   d\mu(\lambda),$$
and 
$$(II):=\int_{\mathbb{R}^*} \sum_{k, \ell\in \mathbb{N}^n} (|\lambda| \mu_k)^{s}  |A_1(t, \lambda)_{k\ell}|^2 |\widehat{u}_1(\lambda)_{k\ell} |^2  d\mu(\lambda).$$

   {  
\noindent \textbf{Case 1: When $|\beta_{k, \lambda}|<\varepsilon \ll 1$:} Recalling that $\beta_{k, \lambda}^2=|\lambda| \mu_k$ and  using \eqref{K_0estimate} and \eqref{decay},  we get 
\begin{align}\label{eq02}\nonumber
	(I)&=  \sum_{k, \ell\in \mathbb{N}^n} \int_{|\lambda|<\frac{\varepsilon^2}{\mu_k}} (|\lambda| \mu_k)^{s+\gamma}   |A_0(t, \lambda)_{k\ell}|^2  (|\lambda| \mu_k)^{-\gamma}|\widehat{u}_0(\lambda)_{kl}|^2   d\mu(\lambda)\\\nonumber
	&	\leq   \sum_{k, \ell\in \mathbb{N}^n} \int_{|\lambda|<\frac{\varepsilon^2}{\mu_k}}  (|\lambda| \mu_k)^{s+\gamma}   \{	|\beta_{k, \lambda}|^2  {e}^{-c t}+ {e}^{-ct\beta_{k, \lambda}^2 }\}^2 (|\lambda| \mu_k)^{-\gamma}|\widehat{u}_0(\lambda)_{kl}|^2   d\mu(\lambda)\\\nonumber
	& \leq \sum_{k, \ell\in \mathbb{N}^n} \int_{|\lambda|<\frac{\varepsilon^2}{\mu_k}}  (|\lambda| \mu_k)^{s+\gamma} {e}^{-2ct\beta_{k, \lambda}^2}  \{	|\beta_{k, \lambda}|^2  {e}^{-c t (1-\beta_{k, \lambda}^2)}+ 1\}^2 (|\lambda| \mu_k)^{-\gamma}|\widehat{u}_0(\lambda)_{kl}|^2   d\mu(\lambda)\\\nonumber
	& \lesssim (1+t)^{-(s+\gamma)}	     \sum_{k, \ell\in \mathbb{N}^n} \int_{|\lambda|<\frac{\varepsilon^2}{\mu_k}}    (|\lambda| \mu_k)^{-\gamma}|\widehat{u}_0(\lambda)_{kl}|^2   d\mu(\lambda)\\
	&\lesssim (1+t)^{-(s+\gamma)}	  \|{u}_0 \|_{\dot H^{-\gamma}_{\mathcal{L}}}^2   .
\end{align}
Similarly, using \eqref{K_1estimate} and \eqref{decay} we can find that 
\begin{align}\label{eq03}\nonumber
	(II)&=\sum_{k, \ell\in \mathbb{N}^n}\int_{|\lambda|<\frac{\varepsilon^2}{\mu_k}}   (|\lambda| \mu_k)^{s}  |A_1(t, \lambda)_{k\ell}|^2 |\widehat{u}_1(\lambda)_{k\ell} |^2  d\mu(\lambda)\\
	&\lesssim (1+t)^{-(s+\gamma)}	  \|{u}_1 \|_{ \dot H^{-\gamma}_{\mathcal{L}}}^2   .
\end{align}

\noindent	\textbf{Case 2: When $|\beta_{k, \lambda}|>N  \gg 1$:} Again using \eqref{K_0estimate} we get
\begin{align}\label{eq04}\nonumber
	(I)&=  \sum_{k, \ell\in \mathbb{N}^n} \int_{|\lambda|>\frac{N^2}{\mu_k}}  (|\lambda| \mu_k)^{s}   |A_0(t, \lambda)_{k\ell}|^2  |\widehat{u}_0(\lambda)_{kl}|^2   d\mu(\lambda)\\\nonumber
	& 	\lesssim  {e}^{-2c t}   \sum_{k, \ell\in \mathbb{N}^n} \int_{|\lambda|>\frac{N^2}{\mu_k}}  (1+|\lambda| \mu_k)^{s }     |\widehat{u}_0(\lambda)_{kl}|^2   d\mu(\lambda)\\
	& \lesssim {e}^{-2c t}   \|{u}_0 \|_{H_\mathcal{L}^{-s}}^2   .
\end{align}
Also, with the help of \eqref{K_1estimate} 
we  find that 
\begin{align}\label{eq05}\nonumber
	(II)&= \sum_{k, \ell\in \mathbb{N}^n} \int_{|\lambda|>\frac{N^2}{\mu_k}} (|\lambda| \mu_k)^{s}  |A_1(t, \lambda)_{k\ell}|^2 |\widehat{u}_1(\lambda)_{k\ell} |^2  d\mu(\lambda)\\\nonumber
	&\lesssim   \sum_{k, \ell\in \mathbb{N}^n} \int_{|\lambda|>\frac{N^2}{\mu_k}} (|\lambda| \mu_k)^{s}  	|\beta_{k, \lambda}|^{-2}  {e}^{-2c t} |\widehat{u}_1(\lambda)_{k\ell} |^2  d\mu(\lambda)\\\nonumber
	&=   {e}^{-2c t}  \sum_{k, \ell\in \mathbb{N}^n} \int_{|\lambda|>\frac{N^2}{\mu_k}} (1+|\lambda| \mu_k)^{s-1}  \left( \frac{|\lambda| \mu_k }{1+|\lambda| \mu_k}\right)^{s-1} 	   |\widehat{u}_1(\lambda)_{k\ell} |^2  d\mu(\lambda)\\
	&\lesssim     {e}^{-2c t}  \|{u}_1 \|_{H_\mathcal{L}^{s-1}}^2.
\end{align}
\textbf{Case 3: When $\varepsilon\leq |\beta_{k, \lambda}|\leq N$:} From \eqref{K_0estimate}  and  \eqref{K_1estimate} , we get

\begin{align}\label{eqs01}\nonumber
	(I)&=\sum_{k, \ell\in \mathbb{N}^n}  \int_{\frac{N^2}{\mu_k}\leq |\lambda|\leq \frac{\varepsilon^2}{\mu_k} }    (|\lambda| \mu_k)^{s}   |A_0(t, \lambda)_{k\ell}|^2  |\widehat{u}_0(\lambda)_{kl}|^2   d\mu(\lambda)\\\nonumber
	& \lesssim {e}^{-2c t}  \sum_{k, \ell\in \mathbb{N}^n} \int_{\frac{N^2}{\mu_k}\leq |\lambda|\leq \frac{\varepsilon^2}{\mu_k} }  (|\lambda| \mu_k)^{s}     |\widehat{u}_0(\lambda)_{kl}|^2   d\mu(\lambda)\\\nonumber
	& \lesssim {e}^{-2c t}  \sum_{k, \ell\in \mathbb{N}^n}     \int_{\frac{N^2}{\mu_k}\leq |\lambda|\leq \frac{\varepsilon^2}{\mu_k} }     |\widehat{u}_0(\lambda)_{kl}|^2   d\mu(\lambda)\\& \lesssim {e}^{-2c t} \|{u}_0 \|_{L^2}^2,
\end{align}
and 
\begin{align}\label{eqs05}\nonumber
	(II)&=  \sum_{k, \ell\in \mathbb{N}^n} \int_{\frac{N^2}{\mu_k}\leq |\lambda|\leq \frac{\varepsilon^2}{\mu_k} }  (|\lambda| \mu_k)^{s}  |A_1(t, \lambda)_{k\ell}|^2 |\widehat{u}_1(\lambda)_{k\ell} |^2  d\mu(\lambda)\\\nonumber
	& \lesssim {e}^{-2c t}  \sum_{k, \ell\in \mathbb{N}^n} \int_{\frac{N^2}{\mu_k}\leq |\lambda|\leq \frac{\varepsilon^2}{\mu_k} }  (1+|\lambda| \mu_k)^{s}    |\widehat{u}_0(\lambda)_{kl}|^2   d\mu(\lambda)\\\nonumber
	& \lesssim {e}^{-2c t}  \sum_{k, \ell\in \mathbb{N}^n} \int_{\frac{N^2}{\mu_k}\leq |\lambda|\leq \frac{\varepsilon^2}{\mu_k} }   (1+|\lambda| \mu_k)^{s-1}     |\widehat{u}_0(\lambda)_{kl}|^2   d\mu(\lambda) 	\\
	& \lesssim  {e}^{-2c t} \|{u}_1 \|_{H_\mathcal{L}^{s-1}}^2.
\end{align}

%
%

\begin{proof}[Proof of Theorem \ref{eq0133}]
	Combining     above Case 1 [\eqref{eq02} and \eqref{eq03}], Case 2 [\eqref{eq04} and \eqref{eq05}], and Case 3 [\eqref{eqs01} and \eqref{eqs05}],  we obtain   
	the following $ \dot {H}_\mathcal{L}^s$-decay estimate	$$
	\|u(t, \cdot)\|_{ \dot{H}_\mathcal{L}^s} \lesssim(1+t)^{-\frac{s+\gamma}{2}}\left(\left\|u_0\right\|_{H^s_\mathcal{L} \cap \dot{H}^{-\gamma}_\mathcal{L}}+\left\|u_1\right\|_{H_\mathcal{L}^{s-1} \cap \dot{H}_\mathcal{L}^{-\gamma}}\right) ,
	$$		for any $t\geq 0$. 
\end{proof}}

\section{Global-in-time well-posedness} \label{sec4}
In this section, we will prove  Theorem \ref{well-posed}, that is,   the global-in-time well-posedness of the   Cauchy problem   \eqref{eq0010}  in the energy evolution space  $\mathcal C\left([0,T],  H^s_{\mathcal{L}}(\mathbb{H}^n)\right)$.  

	\begin{proof}[Proof of Theorem \ref{well-posed}]
		Recall that  for $s \geq 0$ and $\gamma \in \mathbb{R}$ such that $s+\gamma \geq  0$, 	from  the estimate (\ref{hom}) of Theorem \ref{eq0133}, 
		we have 
		\begin{equation} \label{4er}
			\|u(t, \cdot)\|_{\dot{{H}}_\mathcal{L}^s(\HH)} \lesssim(1+t)^{-\frac{s+\gamma}{2}}\left(\left\|u_0\right\|_{H^s_{\mathcal{L}}  \cap \dot {{H}}_\mathcal{L}^{-\gamma} }+\left\|u_1\right\|_{H_\mathcal{L}^{s-1}  \cap \dot{{H}}_\mathcal{L}^{-\gamma} }\right) .
		\end{equation}
		In particular, for $\gamma=0$ and $s \geq 0,$ we get 
		\begin{align}\label{eq15}
			\|u(t, \cdot)\|_{{\dot{H}}_\mathcal{L}^s} \lesssim(1+t)^{-\frac{s}{2}}\left(\left\|u_0\right\|_{H_\mathcal{L}^s }+\left\|u_1\right\|_{L^2}\right).
		\end{align}
		From \eqref{4er}, in particular for $s \in [0, 1]$ and $\gamma>0$ 
		we have the following  estimate for the Sobolev solutions of the linear Cauchy problem
		\begin{align}\label{eq14}
			\|u(t, \cdot)\|_{\dot{{H}}_\mathcal{L}^s} \lesssim(1+t)^{-\frac{s+\gamma}{2}}\left(\left\|u_0\right\|_{ H^s_{\mathcal{L}} \cap {\dot{H}}_\mathcal{L}^{-\gamma}}+\left\|u_1\right\|_{L^2 \cap {\dot{H}}_\mathcal{L}^{-\gamma}}\right), 
		\end{align}
		where we have used the Sobolev embedding   $L^2\subset H_\mathcal{L}^{s-1}$ for $s\leq 1$.
		Particularly,  for $s=0$ in  (\ref{eq14}), using  the Sobolev embedding $ H_\mathcal{L}^{s} \subset L^2$ for $s\geq0$, we obtain \begin{align}\label{eq1444}
			\nonumber	\|u(t, \cdot)\|_{L^2} &\lesssim(1+t)^{-\frac{\gamma}{2}}\left(\left\|u_0\right\|_{ L^2 \cap {\dot{H}}_\mathcal{L}^{-\gamma}}+\left\|u_1\right\|_{L^2 \cap {\dot{H}}_\mathcal{L}^{-\gamma}}\right) \\&\lesssim(1+t)^{-\frac{\gamma}{2}}\left(\left\|u_0\right\|_{ H^s_{\mathcal{L}} \cap {\dot{H}}_\mathcal{L}^{-\gamma}}+\left\|u_1\right\|_{L^2 \cap {\dot{H}}_\mathcal{L}^{-\gamma}}\right).
		\end{align}
		Now from (\ref{eq14}) and   (\ref{eq1444}), for   $  s\in (0, 1]$,  we can write
		\begin{align*}
			(1+t)^{\frac{s+\gamma}{2}}\|u(t, \cdot)\|_{{\dot {H}}_\mathcal{L}^s} &+(1+t)^{\frac{\gamma}{2}}\|u(t, \cdot)\|_{L^2}\\ &\lesssim \left(\left\|u_0\right\|_{H^s_{\mathcal{L}} \cap {\dot{H}}_\mathcal{L}^{-\gamma}}+\left\|u_1\right\|_{L^2 \cap {\dot {H}}_\mathcal{L}^{-\gamma}}\right)=\left\|\left(u_{0}, u_{1}\right)\right\|_{ {  \mathcal{A}_{\mathcal{L}}^{s, -\gamma}}}.
		\end{align*}
		Thus from above, we can   claim that   
		$u \in X_s(T)$ and  \begin{align}\label{2number100}
			\|u^{\text{lin}}\|_{X_s(T)} \lesssim   \left\|\left(u_{0}, u_{1}\right)\right\|_{ {  \mathcal{A}_{\mathcal{L}}^{s, -\gamma}}}.
		\end{align} 
		Our next aim is to prove 
		$$	\|u^{\text{non}}\|_{X_s(T)} \lesssim    \left\| u\right\|_{ X_s(T)}^p,$$
		under some conditions for $p$.
		First we will    {  estimate} $L^2$    and $\dot H^{-\gamma}_{\mathcal{L}}$ norm of $|u(t, \cdot)|^p$.  Applying Gagliardo-Nirenberg  inequality (Theorem \ref{eq16}), we have
		\begin{align}\label{eqq21}\nonumber
			&	\left \||u(\sigma, \cdot)|^p\right\|_{L^2}=	\left \|u(\sigma, \cdot)\right\|_{L^{2p}}^p\\\nonumber
			&\lesssim\|u(\sigma, \cdot)\|_{\dot {H}_\mathcal{L}^s(\mathbb{H}^n)}^{p\theta_1} \|u(\sigma, \cdot)\|_{L^2(\mathbb{H}^n)}^{p(1-\theta_1)}\\\nonumber
			&= (1+\sigma)^{-\frac{p}{2}  (s\theta_1+\gamma)}  \left\{ (1+\sigma)^{\frac{(s+\gamma)}{2} } \|u(\sigma, \cdot)\|_{\dot {H}_\mathcal{L}^s(\mathbb{H}^n)} \right\}^{p\theta_1}  \left\{ (1+\sigma)^{\frac{\gamma}{2} } \|u(\sigma, \cdot)\|_{L^2(\mathbb{H}^n)}\right\}^{p(1-\theta_1)}\\\nonumber&= (1+\sigma)^{-\frac{p}{2}\left[ \gamma+\frac{Q}{2}\left(1-\frac{1}{p}\right) \right ]}  \left\{ (1+\sigma)^{\frac{(s+\gamma)}{2} } \|u(\sigma, \cdot)\|_{\dot {H}_\mathcal{L}^s(\mathbb{H}^n)} \right\}^{p\theta_1}  \left\{ (1+\sigma)^{\frac{\gamma}{2} } \|u(\sigma, \cdot)\|_{L^2(\mathbb{H}^n)}\right\}^{p(1-\theta_1)}\\
			&\lesssim (1+\sigma)^{- p\left(\frac{\gamma}{2}+\frac{Q}{4}\right)+\frac{Q}{4}}    \left\| u\right\|_{ X_s(T)}^p,
		\end{align}
		for $\sigma\in [0, T]$ with $\theta_1=\frac{Q}{2s}(1-\frac{1}{p})\in [0, 1]$, provided that 
		\begin{align}\label{eq18}
			1\leq p\leq \frac{Q}{Q-2s}, \quad \text{if} ~~Q>2s.
		\end{align}
		On the other hand, using the  Hardy-Littlewood-Sobolev inequality   (Theorem \ref{eq177}), we have
		\begin{align*}
			\left \||u(\sigma, \cdot)|^p\right\|_{\dot H^{-\gamma}_{\mathcal{L}}}
			\lesssim	\left \| |u(\sigma, \cdot)|^p\right\|_{L^{m}}
			=	\left \| u(\sigma, \cdot)\right\|_{L^{mp}}^p
		\end{align*}
		with $\frac{1}{m}-\frac{1}{2}=\frac{\gamma}{Q}$ provided that  $0<\gamma <Q$ and $1<m<2$. Since   $m \in(1,2)$, we have to restrict $0<\gamma  <\frac{Q}{2}$. Again,   the Gagliardo-Nirenberg inequality (Theorem \ref{eq16}) implies that 
		\begin{align}\label{eq220}\nonumber
			&	\left \||u(\sigma, \cdot)|^p\right\|_{\dot H^{-\gamma}_{\mathcal{L}}}\lesssim	\left \| u(\sigma, \cdot)\right\|_{L^{mp}}^p\\\nonumber
			&\lesssim\| u(\sigma, \cdot)\|_{\dot {H}_\mathcal{L}^s(\mathbb{H}^n)}^{p\theta_2} \|u(\sigma, \cdot)\|_{L^2(\mathbb{H}^n)}^{p(1-\theta_2)}\\\nonumber
			&= (1+\sigma)^{-\frac{p}{2}  (s\theta_2+\gamma)}  \left\{ (1+\sigma)^{\frac{(s+\gamma)}{2} } \|u(\sigma, \cdot)\|_{\dot {H}_\mathcal{L}^s(\mathbb{H}^n)} \right\}^{p\theta_1}  \left\{ (1+\sigma)^{\frac{\gamma}{2} } \|u(\sigma, \cdot)\|_{L^2(\mathbb{H}^n)}\right\}^{p(1-\theta_1)}\\\nonumber
			&= (1+\sigma)^{-\frac{p}{2}  \left[ \gamma+Q(\frac{1}{2}-\frac{1}{mp})\right] }  \left\{ (1+\sigma)^{\frac{(s+\gamma)}{2} } \|u(\sigma, \cdot)\|_{\dot {H}_\mathcal{L}^s(\mathbb{H}^n)} \right\}^{p\theta_1}  \left\{ (1+\sigma)^{\frac{\gamma}{2} } \|u(\sigma, \cdot)\|_{L^2(\mathbb{H}^n)}\right\}^{p(1-\theta_1)}\\\nonumber
			&= (1+\sigma)^{-p ( \frac{\gamma}{2}+\frac{Q}{4} )+ \frac{\gamma}{2}+\frac{Q}{4}  }  \left\{ (1+\sigma)^{\frac{(s+\gamma)}{2} } \|u(\sigma, \cdot)\|_{\dot {H}_\mathcal{L}^s(\mathbb{H}^n)} \right\}^{p\theta_2}  \left\{ (1+\sigma)^{\frac{\gamma}{2} } \|u(\sigma, \cdot)\|_{L^2(\mathbb{H}^n)}\right\}^{p(1-\theta_2)}\\
			&\lesssim (1+\sigma)^{-p ( \frac{\gamma}{2}+\frac{Q}{4} )+ \frac{\gamma}{2}+\frac{Q}{4}  }  \left\| u\right\|_{ X_s(T)}^p,
		\end{align}
		for $\sigma\in [0, T]$ with $\theta_2=\frac{Q}{s}(\frac{1}{2}-\frac{1}{mp})\in [0, 1]$.     Using the fact that $\theta_2\in [0, 1]$, we get  
		\begin{align}\label{eq17}
			p\geq \frac{2}{m} \quad \text{and}\quad  p\leq   \frac{2Q}{m(Q-2s)}, \quad \text{if} ~Q>2s.
		\end{align}
		Since  $m=\frac{2Q}{Q+2\gamma}\in (1, 2)$,  from  (\ref{eq17}),  we obtain
		\begin{align}\label{eq19}
			1<	1+\frac{2\gamma}{Q} \leq p  
			\leq  {   \frac{Q+2\gamma}{(Q-2s)}}\quad  \text { if } Q>2s.
		\end{align}
		Therefore, from (\ref{eq18}) and (\ref{eq19}), if we consider 
		\begin{align*}
			1+\frac{2\gamma}{Q} \leq p  \left\{\begin{array}{ll}
				<\infty & \text { if } Q \leq 2s, \\
				\leq   \frac{Q}{(Q-2s)}& \text { if } Q>2s,
			\end{array}\right. 
		\end{align*}
		then from (\ref{eqq21}) and (\ref{eq220}),   for $\sigma\in [0, T]$,   we assert that
		\begin{align}\label{eq20}
			\left \||u(\sigma, \cdot)|^p\right\|_{L^2\cap \dot H^{-\gamma}_{\mathcal{L}}}\lesssim (1+\sigma)^{-p ( \frac{\gamma}{2}+\frac{Q}{4} )+ \frac{\gamma}{2}+\frac{Q}{4}  }  \left\| u\right\|_{ X_s(T)}^p.
		\end{align}
		Using the $ (L^2 \cap \dot {H}_\mathcal{L}^{-\gamma} )-L^2$ estimate given in (\ref{eq14}) (with $u_0=0$, $u_1= |u(\sigma,\cdot)|^p$) as well as from the statement (\ref{eq20}), we have the following $L^2$-estimate  of the solution
		\begin{align}\nonumber
			\left\|u^{\mathrm{non}}(t, \cdot)\right\|_{L^2}&=\bigg \|  \int\limits_0^t  |u(\sigma,\cdot )|^p*_{(g)} E_1(t-\sigma, \cdot) d\sigma\bigg\|_{L^2 }\\\nonumber
			&\leq   \int\limits_0^t (1+t-\sigma)^{-\frac{\gamma}{2}}  \| |u(\sigma,\cdot)|^p\|_{L^2 \cap \dot H^{-\gamma}_{\mathcal{L}}}  d\sigma\\\nonumber
			&\lesssim \int_0^t(1+t-\sigma)^{-\frac{\gamma}{2}}(1+\sigma)^{-p ( \frac{\gamma}{2}+\frac{Q}{4} )+ \frac{\gamma}{2}+\frac{Q}{4}  }d\sigma ~  \left\| u\right\|_{ X_s(T)}^p\\\label{eq21}
			& \lesssim(1+t)^{-\frac{\gamma}{2}} \int_0^{\frac{t}{2}}(1+\sigma)^{-p ( \frac{\gamma}{2}+\frac{Q}{4} )+ \frac{\gamma}{2}+\frac{Q}{4}  }d\sigma ~  \left\| u\right\|_{ X_s(T)}^p\\\label{eq222222}
			& \qquad +(1+t)^{-p ( \frac{\gamma}{2}+\frac{Q}{4} )+ \frac{\gamma}{2}+\frac{Q}{4}  } \int_{\frac{t}{2}}^t(1+t-\sigma)^{-\frac{\gamma}{2}} \mathrm{~d} \sigma\|u\|_{X_s(T)}^p .
		\end{align}
		For $p>1+\frac{4}{Q+2 \gamma},$ the  integral given in (\ref{eq21})  converges  uniformly  over $\left[0, \frac{t}{2}\right]$. Now we will compute the integral given in (\ref{eq222222})  precisely.  Now 
		\begin{align}\label{integral}
			\int_{\frac{t}{2}}^t(1+t-\sigma)^{-\frac{\gamma}{2}} \mathrm{~d} \sigma \lesssim\left\{\begin{array}{ll}
				(1+t)^{1-\frac{\gamma}{2}} & \text { if } \gamma<2, \\
				\ln (\mathrm{e}+t) & \text { if } \gamma=2, \\
				1 & \text { if } \gamma>2.
			\end{array}\right.
		\end{align}
		Thus by  considering $p>1+\frac{4}{Q+2 \gamma}$ if $\gamma \leq 2$ and using the integral estimate (\ref{integral}), we  get
		\begin{align*}
			& (1+t)^{-\left(\frac{\gamma}{2}+\frac{Q}{4}\right) p+\frac{\gamma}{2}+\frac{Q}{4}} \int_{\frac{t}{2}}^t(1+t-\sigma)^{-\frac{\gamma}{2}}  {~d} \sigma \\
			& \leq (1+t)^{-1} \int_{\frac{t}{2}}^t(1+t-\sigma)^{-\frac{\gamma}{2}}  {~d} \sigma \lesssim  (1+t)^{-\frac{\gamma}{2}}.
		\end{align*}
		Again by considering  $p>1+\frac{2 \gamma}{Q+2 \gamma}$ if $\gamma>2$ and using the integral estimate (\ref{integral}),  a simple calculation yields
		\begin{align*}
			(1+t)^{-\left(\frac{\gamma}{2}+\frac{Q}{4}\right) p+\frac{\gamma}{2}+\frac{Q}{4}} \int_{\frac{t}{2}}^t(1+t-\sigma)^{-\frac{\gamma}{2}}  {~d} \sigma \lesssim(1+t)^{-\frac{\gamma}{2}} .
		\end{align*}
		Summarizing   all the  estimates, it holds that 
		\begin{align}\label{eq23}
			(1+t)^{\frac{\gamma}{2}} 	\|u^{\text{non}}(t, \cdot )\|_{L^2} \lesssim    \left\| u\right\|_{ X_s(T)}^p. \end{align}
		Now we will estimate  $\dot {H}_\mathcal{L}^s$-norm of the solution. Similar to the previous argument, we use   $\left(L^2 \cap \dot {H}_\mathcal{L}^{-\gamma}\right)-\dot 
		{H}_\mathcal{L}^s$ estimate (\ref{eq14}) in $\left[0, \frac{t}{2}\right]$ and the $L^2-\dot {H}_\mathcal{L}^s$ estimate (\ref{eq15}) in $\left[\frac{t}{2}, t\right]$ to estimate the solution itself in $\dot {H}_\mathcal{L}^s$. Thus 
		\begin{align}\nonumber
			\left\|u^{\mathrm{non}}(t, \cdot)\right\|_{\dot {H}_\mathcal{L}^s}
			&\leq   \int\limits_0^t \|    |u(\sigma,\cdot )|^p*_{(g)} E_1(t-\sigma, \cdot)  \|_{\dot {H}_\mathcal{L}^s} d\sigma\\\nonumber
			& \lesssim \int_0^{\frac{t}{2}}(1+t-\sigma)^{-\frac{(s+\gamma)}{2}} (1+\sigma)^{-p ( \frac{\gamma}{2}+\frac{Q}{4} )+ \frac{\gamma}{2}+\frac{Q}{4}  }d\sigma ~  \left\| u\right\|_{ X_s(T)}^p\\\nonumber
			& \qquad + \int_{\frac{t}{2}}^t (1+t-\sigma)^{-\frac{s}{2}} (1+\sigma)^{-p ( \frac{\gamma}{2}+\frac{Q}{4} )+\frac{Q}{4}  } \mathrm{~d} \sigma\|u\|_{X_s(T)}^p \\\nonumber
			& \lesssim(1+t)^{-\frac{(s+\gamma)}{2}} \int_0^{\frac{t}{2}}(1+\sigma)^{-p ( \frac{\gamma}{2}+\frac{Q}{4} )+ \frac{\gamma}{2}+\frac{Q}{4}  }d\sigma ~  \left\| u\right\|_{ X_s(T)}^p\\\label{eq22}
			& \qquad +(1+t)^{-p ( \frac{\gamma}{2}+\frac{Q}{4} )+\frac{Q}{4}  } \int_{\frac{t}{2}}^t(1+t-\sigma)^{-\frac{s}{2}} \mathrm{~d} \sigma\|u\|_{X_s(T)}^p .
		\end{align}
		For $p>1+\frac{4}{Q+2 \gamma}$,
		summarizing as previously, we have
		\begin{align}\label{eqq22}
			(1+t)^{\frac{s+\gamma}{2}}\left\|u^{ \text{non}}(t, \cdot)\right\|_{\dot {H}_\mathcal{L}^s} \lesssim\|u\|_{X_s(T)}^p .
		\end{align}
		Therefore from (\ref{eqq22}) and (\ref{eq23}), we obtain
		\begin{align*}
			\left\|u^{ \text{non}}(t, \cdot)\right\|_{X_s(T)} \lesssim\|u\|_{X_s(T)}^p,
		\end{align*}
		under the   conditions on $p$ as follows:
		\begin{itemize}
			\item $	1<	  p  
			\leq   \frac{Q}{Q-2s}$, from the application of the Gagliardo-Nirenberg inequality, 
			\item $p>1+\frac{4}{Q+2 \gamma}$ if $\gamma \leq 2$,
			\item $p>1+\frac{2 \gamma}{Q+2 \gamma}$ if $\gamma>2$,  from the integrability and decay estimates of solution.
			\item $p \geq 1+\frac{2 \gamma}{Q}$, from the application of the Gagliardo-Nirenberg inequality. 
		\end{itemize}
		First, we consider the case $\gamma>2.$ Then we observe that	
		$$
		\max \left\{1+\frac{2 \gamma}{Q+2 \gamma}, 1+\frac{2 \gamma}{Q}\right\}=1+\frac{2 \gamma}{Q} .
		$$
		Now for $\gamma\leq 2$, we have to compare between $1+\frac{4}{Q+2 \gamma}$  and $1+\frac{2 \gamma}{Q}$. Notice that  $1+\frac{4}{Q+2 \gamma}$  and $1+\frac{2 \gamma}{Q}$   intersects at a point $\tilde \gamma$, which is the positive root of the quadratic equation $ 2\gamma^2+Q\gamma-2Q=0$.  Moreover, it is easy to check that  the positive root $\tilde  \gamma<2$ for all $Q\geq 4.$ This shows that  
		\begin{itemize}
			\item $\max \left\{1+\frac{4}{Q+2 \gamma}, 1+\frac{2 \gamma}{Q} \right\}=1+\frac{4}{Q+2 \gamma} $ for $\gamma \leq \tilde{\gamma}$,
			\item $\max \left\{1+\frac{4}{Q+2 \gamma}, 1+\frac{2 \gamma}{Q} \right\}=1+\frac{2 \gamma}{Q} $ for  $\tilde{\gamma}<\gamma \leq 2$.
		\end{itemize}
		Finally, for any $\gamma\in (0, \frac{Q}{2})$, the condition for the exponent $p$ is reduced to    \begin{align*} 
			p\left\{\begin{array}{ll}
				>1+\frac{4}{Q+2 \gamma} =p_{\text {crit }}(Q, \gamma) & \text { if } \gamma \leq \tilde{\gamma}, \\
				\geq 1+\frac{2 \gamma}{Q} & \text { if } \gamma>\tilde{\gamma}. 
			\end{array}\right. 
		\end{align*}
		Furthermore, we have 
		\begin{align}\label{Nu}
			\|N u\|_{X_s(T)} \lesssim \left\|\left(u_{0}, u_{1}\right)\right\|_{ {  \mathcal{A}_{\mathcal{L}}^{s, -\gamma}}}+\|u\|_{X_s(T)}^{p}.
		\end{align}
		Now  we will calculate $
		\|N u-N \bar{u}\|_{X_s(T)}$. In order to do that, we first notice that
		$$
		\|N u-N \bar{u}\|_{X_s(T)}=\left\|\int_0^t E_1(t-\sigma, \cdot) *_{(g)}\left(|u(\sigma, \cdot)|^p-|\bar{u}(\sigma, \cdot)|^p\right)  {d} \sigma\right\|_{X_s(T)} .
		$$
		Similar to (\ref{eq20}), 	using the   Hardy-Littlewood-Sobolev inequality     (Theorem \ref{eq177}), we have
		\begin{align*}
			\left \| |u(\sigma, \cdot)|^p-|\bar{u}(\sigma, \cdot)|^p \right\|_{\dot  H^{-\gamma}_{\mathcal{L}}}
			\lesssim	\left \| |u(\sigma, \cdot)|^p-|\bar{u}(\sigma, \cdot)|^p\right\|_{L^{m}}
		\end{align*}
		with $\frac{1}{m}-\frac{1}{2}=\frac{\gamma}{Q}$ provided that  $0<\gamma <Q$ and $1<m<2$. 
		An application of Hölder's inequality yields
		\begin{align}\label{eq25}
			\left\||u(\sigma, \cdot)|^p-|\bar{u}(\sigma, \cdot)|^p\right\|_{L^m} \lesssim\|u(\sigma, \cdot)-\bar{u}(\sigma, \cdot)\|_{L^{m p}}\left(\|u(\sigma, \cdot)\|_{L^{m p}}^{p-1}+\|\bar{u}(\sigma, \cdot)\|_{L^{m p}}^{p-1}\right) .
		\end{align}
		Again,   the Gagliardo-Nirenberg inequality (Theorem \ref{eq16}) implies that 			\begin{align}\label{Final2}\nonumber
			& 	\left \| u(\sigma, \cdot)-\bar{u}(\sigma, \cdot)\right\|_{L^{mp}}\\\nonumber
			&\lesssim\| u(\sigma, \cdot)-\bar{u}(\sigma, \cdot)\|_{\dot {H}_\mathcal{L}^s(\mathbb{H}^n)}^{\theta_3} \|u(\sigma, \cdot)-\bar{u}(\sigma, \cdot)\|_{L^2(\mathbb{H}^n)}^{(1-\theta_3)}\\\nonumber
			&= (1+\sigma)^{-\left( \frac{\gamma}{2}+\frac{s\theta_3}{2}\right)  }  \left\{ (1+\sigma)^{\frac{(s+\gamma)}{2} } \|u(\sigma, \cdot)-\bar{u}(\sigma, \cdot)\|_{\dot {H}_\mathcal{L}^s(\mathbb{H}^n)} \right\}^{\theta_3}  \\\nonumber&\qquad \times \left\{ (1+\sigma)^{\frac{\gamma}{2} } \|u(\sigma, \cdot)-\bar{u}(\sigma, \cdot)\|_{L^2(\mathbb{H}^n)}\right\}^{(1-\theta_3)}\\
			&\lesssim (1+\sigma)^{ -\left( \frac{\gamma}{2}+\frac{s\theta_3}{2}\right) }  \left\| u(\sigma, \cdot)-\bar{u}(\sigma, \cdot)\right\|_{ X_s(T)},
		\end{align}
		and 
		\begin{align}\label{Final1}\nonumber
			&	 	\left \| u(\sigma, \cdot)\right\|_{L^{mp}}^{p-1}\\\nonumber
			&\lesssim\| u(\sigma, \cdot)\|_{\dot {H}_\mathcal{L}^s(\mathbb{H}^n)}^{(p-1)\theta_3} \|u(\sigma, \cdot)\|_{L^2(\mathbb{H}^n)}^{(p-1)(1-\theta_3)}\\\nonumber
			&= (1+\sigma)^{-\frac{(p-1)}{2}  (s\theta_3+\gamma)}  \left\{ (1+\sigma)^{\frac{(s+\gamma)}{2} } \|u(\sigma, \cdot)\|_{\dot {H}_\mathcal{L}^s(\mathbb{H}^n)} \right\}^{(p-1)\theta_3} \\\nonumber& \quad \times  \left\{ (1+\sigma)^{\frac{\gamma}{2} } \|u(\sigma, \cdot)\|_{L^2(\mathbb{H}^n)}\right\}^{(p-1)(1-\theta_3)}\\
			&\lesssim (1+\sigma)^{-\frac{(p-1)}{2}  (s\theta_3+\gamma)} \left\| u\right\|_{ X_s(T)}^{p-1},
		\end{align}
		for $\sigma\in [0, T]$ with $\theta_3=\frac{Q}{s}(\frac{1}{2}-\frac{1}{mp})\in [0, 1]$. 
		Thus from (\ref{eq25}), (\ref{Final2}), and (\ref{Final1}), we finally  have   
		\begin{align*}
			\nonumber   & \left \| |u(\sigma, \cdot)|^p-|\bar{u}(\sigma, \cdot)|^p \right\|_{\dot  H^{-\gamma}_{\mathcal{L}}}  \\
			& 
			\lesssim (1+\sigma)^{-\left(\frac{p-1}{2}+\frac{1}{2}\right) (\gamma+{Q}(\frac{1}{2}-\frac{1}{mp}))} \left\| u(\sigma, \cdot)-\bar{u}(\sigma, \cdot)\right\|_{ X_s(T)}\left(\|u\|_{X_s(T)}^{p-1}+\|\bar{u}\|_{X_s(T)}^{p-1}\right) \\
			& 
			\lesssim (1+\sigma)^{- \frac{p}{2}(1-\frac{1}{p})(\gamma+\frac{Q}{2}) } \left\| u(\sigma, \cdot)-\bar{u}(\sigma, \cdot)\right\|_{ X_s(T)}\left(\|u\|_{X_s(T)}^{p-1}+\|\bar{u}\|_{X_s(T)}^{p-1}\right)\\
			&
			\lesssim(1+\sigma)^{-\left(\frac{\gamma}{2}+\frac{Q}{4}\right) p+\frac{\gamma}{2}+\frac{Q}{4}} \left\| u(\sigma, \cdot)-\bar{u}(\sigma, \cdot)\right\|_{ X_s(T)}\left(\|u\|_{X_s(T)}^{p-1}+\|\bar{u}\|_{X_s(T)}^{p-1}\right). \nonumber
		\end{align*}
		Now using the given range of $p$, we get   
		\begin{align}\label{Banach}
			\|N u-N \bar{u}\|_{X_s(T)} \leq C \|u-\bar{u}\|_{X_s(T)}\left(\|u\|_{X_s(T)}^{p-1}+\|\bar{u}\|_{X_s(T)}^{p-1}\right),
		\end{align}
		for any $u, \bar u \in X_s(T)$ and for some $C>0.$
		Also from (\ref{Nu}), we obtain 	
		\begin{align}\label{Final banach}
			\|N u\|_{X_s(T)} \leq  D\left\|\left(u_{0}, u_{1}\right)\right\|_{ {  \mathcal{A}_{\mathcal{L}}^{s, -\gamma}}}+D\|u\|_{X_s(T)}^{p},
		\end{align}  
		for some $D>0$ with initial data space ${  \mathcal{A}_{\mathcal{L}}^{s, -\gamma}}: =(H^s_{\mathcal{L}}\cap \dot  H^{-\gamma}_{\mathcal{L}}) \times (L^2\cap \dot  H^{-\gamma}_{\mathcal{L}})$. 
		
Therefore, by Banach's fixed point theorem,  there exists a uniquely determined fixed point $u^*$ of the operator $N$, which means $u^*=Nu^* \in X_s(T)$  for all positive $T$.  This fixed point $u^*$ will be  our mild solution to (\ref{eq0010})  on $[0, T]$. This implies that there exists a global (in-time) small data Sobolev solution $u^*$ of the equation $ u^*=Nu^* $ in $ X_s(T)$, which also gives the solution to the semilinear damped wave equation (\ref{eq0010})  and this completes the proof of the theorem.   \end{proof}


	\section{Blow-up analysis of solutions: a test function method}\label{sec5}
%
%
	
	\begin{proof}[Proof of Theorem \ref{blow-up}]
		In order to prove this result, 	we apply the so-called test function method. By contradiction, we assume that there exists a global in time weak solution $u$ to (\ref{eq0010}).  Let us consider two bump functions $\alpha\in \mathcal{C}_0^{\infty}(\mathbb{R}^n)$ and  $\beta\in \mathcal{C}_0^{\infty}(\mathbb{R})$ such that 
		$$\text{$\alpha=1$ on $B_n\left(\frac{1}{2}\right)$ with     $\operatorname{supp} \alpha \subset B_n(1)$
		}$$ and  $$\text{$\beta=1$ on $\left[-\frac{1}{4}, \frac{1}{4}\right]$ with  $\operatorname{supp} \beta \subset(-1,1)$}.$$ If $R>1$ is a parameter, then, we define the test function $\varphi_R \in \mathcal{C}_0^{\infty}\left([0, \infty) \times \mathbb{R}^{2 n+1}\right)$ with separate variables as
		$$
		\varphi_R(t, x, y, \tau) \doteq \beta\left(\frac{t}{R^2}\right) \alpha\left(\frac{x}{R}\right) \alpha\left(\frac{y}{R}\right) \beta\left(\frac{\tau}{R^2}\right), \quad  (t, x, y, \tau) \in[0, \infty) \times \mathbb{R}^{2 n+1}.
		$$
		Let $\mathcal{D}_R:= B_{n}(R) \times B_{n}(R) \times\left[-R^{2}, R^{2}\right]$.  Then using support of $\alpha$ and $\beta$, we can conclude that $\operatorname{supp} \varphi_{R} \subset\left[0, R^{2}\right] \times \mathcal{D}_R$. 
		Furthermore,  a simple calculation yields \begin{align*}
			\partial_t \varphi_R(t, x, y, \tau)&=R^{-2} \beta^{\prime}\left(\frac{t}{R^2}\right) \alpha\left(\frac{x}{R}\right) \alpha\left(\frac{y}{R}\right) \beta\left(\frac{\tau}{R^2}\right), \\
			\partial_t^2 \varphi_R(t, x, y, \tau)&=R^{-4} \beta^{\prime \prime}\left(\frac{t}{R^2}\right) \alpha\left(\frac{x}{R}\right) \alpha\left(\frac{y}{R}\right) \beta\left(\frac{\tau}{R^2}\right).
		\end{align*} Moreover, from the following expression of the sublaplacian
		$$\mathcal{L}=\Delta+\frac{1}{4}(|x|^2+|y|^2) \partial_\tau^2+\sum_{j=1}^n\left(x_j \partial_{y_j \tau}^2-y_j \partial_{x_j \tau}^2\right),$$ we have 
		\begin{align*}
			\mathcal{L}\varphi_{R}(t, x, y, \tau)&= R^{-2} \beta\left(\frac{t}{R^{2}}\right) \Delta \alpha\left(\frac{x}{R}\right) \alpha\left(\frac{y}{R}\right) \beta\left(\frac{\tau}{R^{2}}\right)\\&+R^{-2} \beta\left(\frac{t}{R^{2}}\right) \alpha\left(\frac{x}{R}\right) \Delta \alpha\left(\frac{y}{R}\right) \beta\left(\frac{\tau}{R^{2}}\right) \\
			& +R^{-3} \sum_{j=1}^{n} x_{j} \beta\left(\frac{t}{R^{2}}\right) \alpha\left(\frac{x}{R}\right) \partial_{j} \alpha\left(\frac{y}{R}\right) \beta^{\prime}\left(\frac{\tau}{R^{2}}\right) \\
			& -R^{-3} \sum_{j=1}^{n} y_{j} \beta\left(\frac{t}{R^{2}}\right) \partial_{j} \alpha\left(\frac{x}{R}\right) \alpha\left(\frac{y}{R}\right) \beta^{\prime}\left(\frac{\tau}{R^{2}}\right) \\
			& +\frac{1}{4} R^{-4}\left(|x|^{2}+|y|^{2}\right) \beta\left(\frac{t}{R^{2}}\right) \alpha\left(\frac{x}{R}\right) \alpha\left(\frac{y}{R}\right) \beta^{\prime \prime}\left(\frac{\tau}{R^{2}}\right),
		\end{align*}
		where $\Delta$ denotes the Laplace operator on $\mathbb{R}^{n}$.
		
		From the condition $0 \leq \alpha, \beta \leq 1$, we   can get    $\alpha \leq \alpha^{\frac{1}{p}}$ and $\beta \leq \beta^{\frac{1}{p}}$. Moreover, using the following bounds
		$$
		\begin{aligned}
			&	\left|\partial_j \alpha\right| \lesssim \alpha^{\frac{1}{p}} \quad \quad\text { for any } 1 \leq j \leq n, \\&\left|\partial_j \partial_k \alpha\right| \lesssim \alpha^{\frac{1}{p}} \quad \text { for any } 1 \leq j, k \leq n, \\
			&	\left|\beta^{\prime}\right|  \lesssim \beta^{\frac{1}{p}}, \quad\left|\beta^{\prime \prime}\right| \lesssim \beta^{\frac{1}{p}} ,
		\end{aligned}
		$$
		we immediately get
		$$
		\begin{aligned}
			\left|\partial_{t} \varphi_{R}\right| & \lesssim R^{-2}\left(\varphi_{R}\right)^{\frac{1}{p}}, \\
			\left|\partial_{t}^{2} \varphi_{R}\right| & \lesssim R^{-4}\left(\varphi_{R}\right)^{\frac{1}{p}} \lesssim R^{-2}\left(\varphi_{R}\right)^{\frac{1}{p}}, \text{and}\\
			\left|\mathcal{L}\varphi_{R}\right| & \lesssim R^{-2}\left(\varphi_{R}\right)^{\frac{1}{p}} .
		\end{aligned}
		$$
		Let
		$$I_R=	\int_0^T \int_{\HH}|u(t, \eta)|^p \varphi_R(t, \eta) \;d \eta \;d t.$$
		Now  from (\ref{2999}) with keeping in mind the fact  $\partial_t \varphi_R \leq 0$    {  (by choosing the factor $\beta$  appropriately, for example, let $\beta$ be even and non-increasing in $[0, \infty)$)}, an application of  H\"older's and Young’s inequalities  yields
		\begin{align}\nonumber\label{eq30}
			&	\int_0^T \int_{\HH}|u(t, g)|^p \varphi_R(t, g) \;d g \;d t+ \varepsilon \int_{\mathbb{H}^n}\left(u_0(g)+u_1(g)\right) \varphi_R(0, g) \;d g\\\nonumber
			&\leq \int_0^T \int_{\HH} |u(t, g)|\left(|\partial_t^2 \varphi_R(t, g)|+|\mathcal{L}\varphi_R(t, g)|+|\partial_t \varphi_R(t, g)|\right) \;d g \;d t\\\nonumber
			&\lesssim R^{-2} \int_0^T \int_{\HH} |u(t, g)||\varphi_R(t, g)|^{\frac{1}{p}} \;dg\;d t\\
			& \leq C R^{-2}\left(\int_{0}^{\infty} \int_{\mathbb{H}^{n}}|u(t, g)|^{p} \varphi_{R}(t, g) \;d g \;d t\right)^{\frac{1}{p}}\left(\int_{\left[0, R^{2}\right] \times \mathcal{D}_R } dt\;dg \right)^{\frac{1}{p^{\prime}}}\\\nonumber
			& = CR^{-2} I_R^{\frac{1}{p}} R^{\frac{2+Q}{p'}}=C R^{-2+\frac{Q+2}{p'}}I_R^{\frac{1}{p}}\\\nonumber
			&\leq \frac{I_R}{p}+\frac{ C R^{Q+2-2p'}}{p'},
		\end{align}
		where  in (\ref{eq30}), we used the fact that  $\operatorname{meas} (\mathcal{D}_R) \approx R^Q$. 
		Therefore 
		\begin{align}\label{eq31}
			\int_0^T \int_{\HH}|u(t, g)|^p \varphi_R(t, g) \;d g \;d t\leq  \frac{I_R}{p}+\frac{CR^{Q+2-2p'}}{p'}- \varepsilon \int_{\mathbb{H}^n}(u_0(g)+u_1(g) ) \varphi_R(0, g) \;d g.
		\end{align} 	
		Now using our assumption (\ref{eq32}), it remains to find estimate for $$	\int_{\HH}\left(u_{0}(g)+u_{1}(g)\right) \varphi_{R}(0, g)dg.$$ Before that, we have to justify that the set of all initial data $\left(u_0, u_1\right)$   in $ \dot {H}_\mathcal{L}^{-\gamma} \times  \dot H_\mathcal{L}^{-\gamma}$  with the assumptions (\ref{eq32}), is non-empty.  For that,  we denote the  set $\mathcal{D}_{Q, \gamma}$ as 
		\begin{align*}
			\mathcal{D}_{Q, \gamma}&:=\{\left(u_{0}, u_{1}\right) \in {  L^1_{\text{loc}}(\mathbb{H}^n) \times L^1_{\text{loc}}(\mathbb{H}^n)}: u_0(g)+u_1(g) \geq C_0 \langle g \rangle^{-Q\left(\frac{1}{2}+\frac{\gamma}{Q}\right)}(\log (e+|g|))^{-1}\},
		\end{align*}
		then first we claim that $\mathcal{D}_{Q, \gamma} \cap\left(L^{\frac{2Q}{Q+2\gamma}} \times L^{\frac{2Q}{Q+2\gamma}}  \right) \neq \emptyset$ for $\frac{2Q}{Q+2\gamma}>1$. In particular, consider the functions	  
		\begin{align*}
			u_{0}(g)=u_{1}(g)&=  C_0 \langle g \rangle^{-Q\left(\frac{1}{2}+\frac{\gamma}{Q}\right)}(\log (e+|g|))^{-1}.
		\end{align*}
		Then using the polar decomposition  {  (see Proposition (1.15) in \cite{Folland})}  for $\mathbb{H}^n$, we have 
		$$
		\begin{aligned}
			\int_{\HH}|u_{0}(g)| ^{\frac{2Q}{Q+2\gamma}}  {d} g& =  C_0^{\frac{2Q}{Q+2\gamma}}  
			\int_{\HH}\langle g \rangle^{-Q\left(\frac{1}{2}+\frac{\gamma}{Q}\right)\times \frac{2Q}{Q+2\gamma}}(\log (e+|g|))^{-\frac{2Q}{Q+2\gamma}}\;dg\\ 
			& =  C_0^{\frac{2Q}{Q+2\gamma}}  \int_{\HH}\langle g \rangle^{-Q }(\log (e+|g|))^{-\frac{2Q}{Q+2\gamma}}\;dg\\ 
			&  \lesssim    C_0^{\frac{2Q}{Q+2\gamma}} \int_{0}^{\infty} \langle r \rangle^{-Q} r^{Q-1}(\log (\mathrm{e}+ r))^{-\frac{2Q}{Q+2\gamma}}dr\\
			& =  C_0^{\frac{2Q}{Q+2\gamma}} \int_{0}^{\infty} \langle r \rangle^{-1}  (\log (\mathrm{e}+ r))^{-\frac{2Q}{Q+2\gamma}}dr<\infty,
		\end{aligned}
		$$ for $\frac{2Q}{Q+2\gamma}>1$.
		This implies that  $u_{0}, u_{1} \in L^{\frac{2Q}{Q+2\gamma}} $. Since $\frac{Q+2 \gamma}{2 Q}-\frac{1}{2}=\frac{\gamma}{Q}$ with $\gamma \in\left(0, \frac{Q}{2}\right)$, according to the Hardy-Littlewood-Sobolev inequality (Theorem \ref{eq177}), we have $ L^{\frac{2Q}{Q+2\gamma}}  \subset \dot {H}_\mathcal{L}^{-\gamma}$.  Therefore, in every instance, we can deduce that    	
		$$
		\mathcal{D}_{Q, \gamma} \cap\left(\dot {H}_\mathcal{L}^{-\gamma}\times \dot {H}_\mathcal{L}^{-\gamma} \right) \neq \emptyset, \quad \text{for} ~\gamma \in\left(0, \frac{Q}{2}\right).
		$$	
		Now from our assumption (\ref{eq32}), for $R \gg 1$,   we obtain 
		\begin{align}\label{eq33}\nonumber
			&	\int_{\HH}\left(u_{0}(g)+u_{1}(g)\right) \varphi_{R}(0, g)\;dg \\\nonumber
			& \geq C_0 \int_{\mathcal{D}_R}    \langle g \rangle^{-Q\left(\frac{1}{2}+\frac{\gamma}{Q}\right)}(\log (e+|g|))^{-1} \,dg \\
			& \geq C_0 R^{Q-Q(\frac{1}{2}+\frac{\gamma}{Q} )}(\log R)^{-1}= C_0 R^{\frac{Q}{2}-\gamma} (\log R)^{-1}.
		\end{align}	
		Thus from (\ref{eq31}) and (\ref{eq33}), we have
		\begin{align*} 
			I_R=	\int_0^T \int_{\HH}|u(t, g)|^p \varphi_R(t, g) \;d g \;d t\leq  \frac{I_R}{p}+\frac{CR^{Q+2-2p'}}{p'}- \varepsilon C_0 R^{\frac{Q}{2}-\gamma}(\log R)^{-1} ,
		\end{align*} 
		which further implies that
		\begin{align}\label{eq34} 
			0\leq \Big(1-\frac{1}{p}\Big) I_R   \leq   \frac{CR^{Q+2-2p'}}{p'}- \varepsilon C_0 R^{\frac{Q}{2}-\gamma}(\log R)^{-1}.
		\end{align} 
		By the assumption $p \in \left(1, 1+\frac{4}{Q+2\gamma}\right),$ we have $Q+2-2p'-\frac{Q}{2}+\gamma <0$ and therefore,  it is easy to see  that 
		\begin{equation}\label{bhah}
			R^{Q+2-2p'-\frac{Q}{2}+\gamma} \log R < \frac{\varepsilon C_0 p'}{C},
		\end{equation} for $R \gg 1.$ 
		
		Hence,  from \eqref{eq34} and \eqref{bhah}, we obtain 
		
		$$0\leq \Big(1-\frac{1}{p}\Big) I_R   \leq   \frac{CR^{Q+2-2p'}}{p'}- \varepsilon C_0 R^{\frac{Q}{2}-\gamma}(\log R)^{-1}<0,$$
		which is a contradiction. This completes the proof of the blow-up result.
		
		Since the scaling  factor $R^2$, appeared   in the bump function $\beta$ with respect to the time varibale   has to be dominated by the lifespan $T_{w,\varepsilon}$ of the weak solution  in order to guarantee $\varphi_R \in \mathcal{C}_0^{\infty}([0, T) \times \mathbb{H}^n)$,  to generate upper bound estimate for the lifespan, we consider $R \uparrow T_{w,\varepsilon}^{ \frac{1}{2}}$ in (\ref{eq34}).  As a result, a contradiction (similar to \eqref{eq34}) exists if $$	\frac{T_{w,\varepsilon}^{\frac{Q+2}{2}-p'}}{p'}<C \varepsilon T_{w,\varepsilon}^{\frac{Q-2\gamma}{4}}(\log T_{w,\varepsilon})^{-1}.	$$		In other words		$$	T_{w,\varepsilon} \leq C\varepsilon^{-\left(\frac{1}{p-1}-\left(\frac{Q}{4}+\frac{\gamma}{2}\right)\right)^{-1}},	$$ which is the desired lifespan for the local in time weak solutions to the Cauchy problem  (\ref{eq0010}), where the constant  $C$ is positive and independent of $\varepsilon $ and $p $.

	\end{proof}

	\section{Sharp lifespan estimates of solutions}\label{sec6}
	The aim of this section is to find the lower bound estimates of lifespan.    We will make use of certain notations, specifically, the evolution space $X_1(T)$ and the data space $\mathcal{A}_1^{\mathcal{L}}$, introduced in Section \ref{sec1}. As we are now considering the case where $1 < p < p_{\text{crit}}(Q, \gamma)$, in order to obtain lower limit estimates for the lifespan,  we will use a different nonlinear inequality instead of (\ref{2number100}).
	
	\begin{proof}[Proof of Theorem \ref{lower bound}]
		Here we note that using the definition of mild solutions (\ref{f2inr}), we can estimate and express local-in-time mild solutions directly. Proceeding in a similar way as to establish  (\ref{eq21})-(\ref{eq222222}) in Section \ref{sec4}, we   get
		\begin{align}\label{eq27}\nonumber
			(1+t)^{\frac{\gamma}{2}}\|u(t, \cdot)\|_{L^2} \lesssim & \varepsilon\left\|\left(u_0, u_1\right)\right\|_{\mathcal{A}_1^{\mathcal{L}}}+   \int_0^{\frac{t}{2}}(1+\sigma)^{-p ( \frac{\gamma}{2}+\frac{Q}{4} )+ \frac{\gamma}{2}+\frac{Q}{4}  }d\sigma ~  \left\| u\right\|_{ X_1(T)}^p\\ 
			&   +(1+t)^{-p ( \frac{\gamma}{2}+\frac{Q}{4} )+ \frac{\gamma}{2}+\frac{Q}{4}  +1}  \|u\|_{X_1(T)}^p
		\end{align}
		because of $\gamma \in (0, 2 )$, where we restricted   $p$ as 
		$$
		1+\frac{2 \gamma}{Q} \leq p
		\leq \frac{Q}{Q-2},
		$$
		for the application of the Gagliardo-Nirenberg inequality.  
		Since $
		1<p<p_{\text {crit }}(Q, \gamma)=1+\frac{4}{Q+2\gamma}$, this implies that  $$-p \left( \frac{\gamma}{2}+\frac{Q}{4} \right)+   \frac{\gamma}{2}+\frac{Q}{4} > -  \frac{2\gamma+Q}{4}   \left(1+\frac{4}{Q+2\gamma}\right) + \frac{2\gamma+Q}{4}=-1.
		$$
		Therefore, from (\ref{eq27}), we have 
		\begin{align}\label{eq28}
			(1+t)^{\frac{\gamma}{2}}\|u(t, \cdot)\|_{L^2} \lesssim \varepsilon\left\|\left(u_0, u_1\right)\right\|_{\mathcal{A}_1^{\mathcal{L}}}+(1+t)^{-p\left(\frac{\gamma}{2}+\frac{Q}{4}\right) +\frac{Q}{4}+\frac{\gamma}{2}+1}\|u\|_{X_1(T)}^p .
		\end{align}
		Choosing $s=1$ in  (\ref{eq22}) leads to
		\begin{align}\label{eq29}
			(1+t)^{\frac{1+\gamma}{2}}	\left\|u(t, \cdot)\right\|_{\dot {H}_\mathcal{L}^1}
			\lesssim \varepsilon\left\|\left(u_0, u_1\right)\right\|_{\mathcal{A}_1^{\mathcal{L}}}
			+(1+t)^{-p ( \frac{\gamma}{2}+\frac{Q}{4} )+\frac{Q}{4}  +\frac{\gamma}{2}+1} \|u\|_{X_1(T)}^p .
		\end{align}
		Thus from (\ref{eq28}) and (\ref{eq29}), we can  say that 
		\begin{align}\label{eq2999}
			\|u\|_{X_1(T)} \leq  \varepsilon C+D (1+t)^{\alpha(p, Q, \gamma)}\|u\|_{X_1(T)}^p,
		\end{align}
		where $\alpha(p, Q, \gamma):=-p ( \frac{\gamma}{2}+\frac{Q}{4} )+\frac{Q}{4}  +\frac{\gamma}{2}+1\in (0,1)$, and $C , D$ are two positive constants independent of $\varepsilon$ and $T$.
		
		Now 	let us now introduce 
		$$
		T^*:=\sup \left\{T \in\left[0, T_{m,\varepsilon}\right) \text { such that } \mathcal{G}(T):=\|u\|_{X_1(T)} \leq M \varepsilon\right\}
		$$
		with a large enough constant $M>0$, which we shall select later. As a result, from (\ref{eq2999}) and the   fact that  $\mathcal{G}\left(T^*\right) \leq M \varepsilon$,	  we obtain 
		\begin{align}\label{eq299}\nonumber
			\mathcal{G}\left(T^*\right)&=\|u\|_{X_1(T^*)} \\\nonumber
			& \leq \left(\varepsilon C+D\left(1+T^*\right)^{\alpha(p, Q, \gamma)} M^p \varepsilon^{p}\right)\\ 
			& = \varepsilon \left( C+D\left(1+T^*\right)^{\alpha(p, Q, \gamma)} M^p \varepsilon^{p-1}\right)	
			<M \varepsilon
		\end{align}
		for large $M$,   provided  {  $2C<M$  and 
		$$
		2D\left(1+T^*\right)^{\alpha(p,Q, \gamma)} M^{p-1} \varepsilon^{p-1}<1 .
		$$}
		Its important to note that the function $\mathcal{G}=\mathcal{G}(T)$ is continuous for  $T \in\left(0, T_{m,\varepsilon}\right)$. Nonetheless, (\ref{eq299}) demonstrates that there exists a time $T_0 \in\left(T^*, T_{m,\varepsilon} \right)$ such that $\mathcal{G}\left(T_0\right) \leq M \varepsilon$, which contradicts to the assumption that  $T^*$ is the supremum.To put it another way, we must enforce the following condition:
		$$
		D\left(1+T^*\right)^{\alpha(p, Q, \gamma)} M^{p-1} \varepsilon^{p-1} \geq 1.
		$$
		This implies that 
		$$
		1+T^*   \geq D^{-\frac{1}{\alpha(p, Q, \gamma)}} M^{-\frac{p-1}{\alpha(p, Q, \gamma)}}\varepsilon^{-\frac{p-1}{\alpha(p, Q, \gamma)}}.
		$$
		This implies that we can deduce the blow-up time estimate  as 
		$$
		T_{m,\varepsilon} \geq D \varepsilon^{-\left(\frac{1}{p-1}-\left(\frac{Q}{4}+\frac{\gamma}{2}\right)\right)^{-1}}.
		$$
		This completes the proof of the theorem regarding lower bound estimates of lifespan. \end{proof}
	
{  \begin{rem}
	 We believe that   the results obtained 
	 in this paper on the Heisenberg group   can be generalized on a general stratified Lie group $\mathbb{G}$, where the formula for $p_{\text {crit }}(Q, \gamma)$ will be the same as in the case of Heisenberg group, with $Q$ now the homogeneous dimension of $\mathbb{G}$. The global existence result can be proved by following the approach of   \cite[Section 4]{30},   while the blow-up result can be proved using again the test function method considered in   \cite{Lie group}.

\end{rem}}
	%
	%
	%
	%
	%
	\section*{Acknowledgement}
{  	The authors wish to thank the anonymous referee for his/her helpful comments and suggestions that helped to improve the quality of the paper.} AD is supported by  Core Research Grant, RP03890G,  Science and Engineering
	Research Board (SERB), DST,  India. VK and MR are supported by the FWO Odysseus 1 grant G.0H94.18N: Analysis and Partial Differential Equations, the Methusalem program of the Ghent University Special Research Fund (BOF) (Grant number 01M01021), and by FWO Senior Research Grant G011522N.  
	MR is also supported by EPSRC grants EP/R003025/2 and EP/V005529/1. SSM is supported by  post-doctoral fellowship at the  Indian Institute of Technology Delhi, India. SSM also thanks Ghent Analysis \& PDE Center of Ghent University for the financial support of his visit to Ghent University during which this work has been completed.

	%
	%
	%
	%
	%
	%
	%
	%
	%
	%
	%


\begin{thebibliography}{99}
		
		
		\baselineskip=11pt
		
		
		
		
		
		
		\bibitem{AZ} J.-Ph. Anker and H.-W. Zhang, Wave equation on general noncompact symmetric spaces, (to appear in) {\it Amer. J. Math}, (2024).
		\url{ https://doi.org/10.48550/arXiv.2010.08467}.
		
		\bibitem{AP} J.-Ph. Anker and V. Pierfelice, Wave and Klein-Gordon equations on hyperbolic spaces, {\it Anal. PDE}, 7(4), 953–995, (2014).
		
		\bibitem{BKM22} A. K. Bhardwaj, V. Kumar and S. S. Mondal, Estimates for the nonlinear viscoelastic damped wave equation on compact Lie groups, \emph{Proc. Roy. Soc. Edinburgh Sect. A.} (2023). \url{https://doi.org/10.1017/prm.2023.38}  
		
		
		\bibitem{DKR23} D. Cardona, V. Kumar and M. Ruzhansky, $L^p$-$L^q$ boundedness of pseudo-differential operators on graded Lie groups, arxiv preprint (2023), \url{https://doi.org/10.48550/arXiv.2307.16094 }
		
		
		\bibitem{Reissig} 	W.  Chen and M.  Reissig,   On the critical exponent and sharp lifespan estimates for semilinear damped wave equations with data from Sobolev spaces of negative order, \emph{J. Evol. Equ.} 23, 13 (2023).  
		
		
		\bibitem{DKM23} A. Dasgupta, V. Kumar and S. S. Mondal, Nonlinear fractional damped wave equations on compact Lie groups, \emph{Asymptot. Anal.} (2023). \url{https://doi.org/10.3233/ASY-231842} 
		
		
		\bibitem{Rei}  M. R. 	Ebert and M.  Reissig,  Methods for Partial Differential Equations, Birkh\"auser, Basel (2018).
		
		
		
		\bibitem{Fischer}	V. Fischer and  M. Ruzhansky, \emph{Quantization on Nilpotent Lie Groups}, Progr. Math., vol.314, Birkh\"auser/Springer, (2016).
		
		{  	\bibitem{Folland} 	G.B. Folland and E.M. Stein, Hardy spaces on homogeneous groups, Princeton University	Press, Princeton, New Jersey (1982).}
		
		\bibitem{garetto} 	C.	 Garetto and M.  Ruzhansky, Wave equation for sums of squares on compact Lie groups, \emph{J. Differential Equations} 258(12),  4324-4347 (2015).		
		
		{  			\bibitem{Lie group}	V. Georgiev and A. Palmieri, Upper bound estimates for local in time solutions to the semilinear heat equation on stratified lie groups in the sub-Fujita case, AIP Conference Proceedings 2159, 020003 (2019).}
			
		\bibitem{Vla}	 V. Georgiev and A. Palmieri, Critical exponent of Fujita-type for the semilinear damped wave equation on the Heisenberg group with power nonlinearity, \emph{J. Differential Equations} 269(1), 420-448 (2020).
		
		\bibitem{palmieri} 	 V.  Georgiev and  A. Palmieri,	Lifespan estimates for local in time solutions	to the semilinear heat equation on the Heisenberg group, \emph{Ann. Mat. Pura Appl.} 200, 999-1032 (2021).
		
		
		
		
		\bibitem{GKR} S. Ghosh, V. Kumar and M. Ruzhansky, Compact Embeddings, Eigenvalue Problems, and subelliptic Brezis-Nirenberg equations involving singularity on stratified Lie groups, \emph{Math. Ann.} (2023). arxiv preprint. \url{https://doi.org/10.1007/s00208-023-02609-7 }
		
		
		\bibitem{GKR2} S. Ghosh, V. Kumar and M. Ruzhansky, Best constants in subelliptic fractional  Sobolev and Gagliardo-Nirenberg inequalities and ground states on stratified Lie groups, (2023). \url{https://doi.org/10.48550/arXiv.2306.07657 }
		
		
		\bibitem{Ikeda2019} M. Ikeda, T. Inui, M. Okamoto and  Y. Wakasugi, $L^p-L^q$ estimates for the damped wave equation and the critical exponent for the nonlinear problem with slowly decaying data, \emph{Commun. Pure Appl. Anal.} 18(4), 1967-2008 (2019).
		
		
		\bibitem{Ikeda16} 	M. Ikeda and T. Ogawa, Lifespan of solutions to the damped wave equation with a critical nonlinearity, \emph{J. Differential Equations} 261(3), 1880-1903 (2016).
		
		
		\bibitem{Ikeda19}  	M. Ikeda and M. Sobajima, Sharp upper bound for the lifespan of solutions to some critical semilinear parabolic, dispersive, and hyperbolic equations via a test function method, \emph{Nonlinear Anal.} 182, 57-74 	(2019).
		
		
		\bibitem{Ikeda15} 	 M. Ikeda and Y. Wakasugi, A note on the lifespan of solutions to the semilinear damped wave equation, \emph{Proc. Amer. Math. Soc.} 143(1), 163-171 (2015).
		
		
		
		\bibitem{Ikeda2002} R. Ikehata and M. Ohta, Critical exponents for semilinear dissipative wave equations in $\mathbb{R}^N$, \emph{J. Math. Anal. Appl.} 269(1), 87-97 (2002).	
		
		
		\bibitem{IKeta and Tanizawa}  R. Ikehata and K. Tanizawa, Global existence of solutions for semilinear damped wave equations in $\mathbb{R}^n$ with noncompactly supported initial data, \emph{Nonlinear Anal.} 61(7), 1189-1208  (2005).
		
		
		\bibitem{AKR22} A. Kassymov, V. Kumar and M. Ruzhansky,  Functional inequalities on symmetric spaces of noncompact type and applications, (2022). arxiv preprint. \url{https://doi.org/10.48550/arXiv.2212.02641} 
		
		
		\bibitem{Lai19}  N. A. Lai and Y. Zhou, The sharp lifespan estimate for semilinear damped wave equation with Fujita critical power in higher dimensions, \emph{J. Math. Pures Appl.} 123, 229-243 (2019).
		
		
	{  	\bibitem{Li1995} 	T. T. Li and Y. Zhou, Breakdown of solutions to $\square u+u_t=|u|^{1+\alpha}$, \emph{Discrete Contin. Dynam. Systems}, 1(4), 503-520 (1995).}
		
		\bibitem{Matsumura}  A. Matsumura, On the asymptotic behavior of solutions of semi-linear wave equations, \emph{Publ. Res. Inst. Math. Sci.} 12(1), 169-189 (1976/77).
		
		
		
		\bibitem{24} A. I. Nachman, The wave equation on the Heisenberg group, \emph{Comm. Partial Differential Equations} 7(6),   675-714 (1982).
		
		
		\bibitem{Nakao93}  M. Nakao and  K. Ono, Existence of global solutions to the Cauchy problem for the semilinear dissipative wave equations, \emph{Math. Z.} 214(2), 325-342 (1993).
		
		
		\bibitem{27}   A. Palmieri,  On the blow-up of solutions to semilinear damped wave equations with power nonlinearity in compact Lie groups, \emph{J. Differential Equations } 281, 85-104 (2021). 
		
		{  		\bibitem{Palmieri 2020}	A. Palmieri, Decay estimates for the linear damped wave equation on the
		Heisenberg group, \emph{J. Funct. Anal.}, 279(9), 108721 (2020). }
		
		\bibitem{31}      A.  Palmieri, Semilinear wave equation on compact Lie groups, \emph{J. Pseudo-Differ. Oper. Appl.} 12, 43 (2021).
		
		
		\bibitem{28}  A.  Palmieri,  A global existence result for a semilinear wave equation with lower order terms on compact Lie groups,  \emph{J. Fourier Anal. Appl. }28,  Article number: 21 (2022).
		
		
		\bibitem{Poho}	 S. I. Pohozaev and  L. V\'eron,  Non existence results of solutions of semilinear differential inequalities  on the Heisenberg group, \emph{Manuscripta Math.} 102,85-99 (2000).
		
		
		\bibitem{gra1}   M. Ruzhansky and C. Taranto, Time-dependent wave equations on graded groups, \emph{Acta Appl. Math.} 171, Article number: 21 (2021).			
		
		
		\bibitem{30} M. Ruzhansky and N. Tokmagambetov, Nonlinear damped wave equations for the sub-Laplacian on the Heisenberg group and for Rockland operators on graded Lie groups, \emph{J. Differential Equations} 265(10), 5212-5236 (2018).
		
		
		
		\bibitem{Nurgi} 		M. Ruzhansky  and N. Yessirkegenov,	Existence and non-existence of global solutions for semilinear heat equations and inequalities on	sub-Riemannian manifolds, and Fujita exponent on unimodular Lie groups, \emph{J. Differential Equations} 308, 455-473 (2022).
		
		
		\bibitem{gra3}   C. Taranto, \emph{Wave equations on graded groups and hypoelliptic Gevrey spaces}, Imperial College London Ph.D. thesis, (2018). 
		
		
		\bibitem{thanga} 	S. Thangavelu,  \emph{Harmonic analysis on the Heisenberg group}, volume 159, Progress in Mathematics, Springer (1998).
		
		
		\bibitem{Todorova}  G. Todorova and  B. Yordanov, Critical exponent for a nonlinear wave equation with damping, \emph{J. Differential Equations} 174(2), 464-489 (2001).
		
		
		\bibitem{HWZ} H.-W. Zhang, Wave equation on certain noncompact symmetric spaces, \emph{Pure Appl. Anal.} 3, 363-386, (2021).
		
		
		\bibitem{Zhang} 	 Q. S. Zhang, A blow-up result for a nonlinear wave equation with damping: the critical case, \emph{C. R. Acad. Sci. Paris Sér. I Math.} 333(2), 109-114  (2001). 
		
		
	\end{thebibliography}
\end{document}